\documentclass[a4paper, 12pt]{amsart}
\usepackage{amsmath,amsthm,amssymb,enumerate}
\usepackage{multicol}

\usepackage[matrix,arrow,curve,frame]{xy}
\usepackage{latexsym}
\usepackage{amscd}
\usepackage{amssymb}
\usepackage{stmaryrd}
\usepackage{mathtools}
\usepackage{mathrsfs}
\usepackage{braket}
\usepackage{fullpage}
\usepackage{xcolor}
\usepackage[shortlabels]{enumitem}
\usepackage{multirow}
\usepackage{adjustbox}

\usepackage{multicol}
\usepackage{booktabs}
\usepackage{makecell}
\allowdisplaybreaks

\usepackage{tikz-cd}
\usepackage{tikz} 
\usepackage{rotating}
\usepackage{enumitem}
\newlist{noitemize}{itemize}{1}

\DeclareMathOperator{\rk}{rk}
\DeclareMathOperator{\diag}{diag}
\DeclareMathOperator{\Mat}{Mat}
\DeclareMathOperator{\Gr}{Gr}
\DeclareMathOperator{\im}{Im}
\DeclareMathOperator{\SL}{SL}

\DeclareMathOperator{\End}{End}

\newtheorem{theorem}{Theorem}[section]
\newtheorem{prop}[theorem]{Proposition}
\newtheorem{lemma}[theorem]{Lemma}
\newtheorem{example}[theorem]{Example}

\theoremstyle{definition}

\newtheorem{remark}[theorem]{Remark}

\begin{document}

	\title{Nilpotent varieties in symmetric spaces and twisted affine Schubert varieties}
	\author{Jiuzu Hong}
\address{Department of Mathematics, University of North Carolina at Chapel Hill, Chapel Hill, NC 27599-3250, U.S.A.}
\email{jiuzu@email.unc.edu}
\author{Korkeat Korkeathikhun}
\address{Department of Mathematics, National University of Singapore, 119076, Singapore}
\email{korkeatk@nus.edu.sg, korkeat.k@gmail.com}

	\maketitle

\begin{abstract}
We relate the geometry of Schubert varieties in twisted affine Grassmannian and the nilpotent varieties in symmetric spaces. This  extends some results of Achar--Henderson in the twisted setting. We also get some applications to the geometry of the order 2 nilpotent varieties in certain classical symmetric spaces. 
\end{abstract}

\section{Introduction}
Let $G$ be a reductive group over $\mathbb{C}$.  Let $\mathcal{N}$ denote the nilpotent cone of the Lie algebra $\mathfrak{g}$ of $G$. Let $\Gr_G$ be the affine Grassmannian of $G$.   Each spherical Schubert cell $\Gr_\lambda$ is parametrized by a dominant coweight $\lambda$.   When $G=\mathrm{GL}_n$,  Lusztig \cite{Lu} defined an embedding from $\mathcal{N}$ to  $\Gr_G$, and showed that each nilpotent variety in $\mathfrak{gl}_n$ can be openly embedded into certain affine Schubert variety $\overline{\Gr}_\lambda$.  This embedding identifies the geometry of nilpotent varieties and certain affine Schubert varieties in type $A$.  However, there is no direct generalization for general reductive groups. 

In \cite{AH},  Achar--Henderson took a different idea for a general algebraic simple group $G$.  Let $\Gr_0^-$ be the opposite open Schubert cell in $\Gr_G$. One can naturally define a map $\pi: \Gr_0^-\to \mathfrak{g}$. Achar--Henderson showed that $\pi( \Gr_0^-  \cap  \Gr_\lambda)$ is contained in $\mathcal{N}$ if and only if $\lambda$ is small in the sense of Broer \cite{Br} and Reeder \cite{Re}, i.e. $\lambda \nsucceq  2\gamma_0$, where $\gamma_0$ is the highest short coroot of $G$. They also proved that $\pi: \Gr_{\rm sm}\cap \Gr_0^-\to  \pi ( \Gr_{\rm sm}\cap \Gr_0^-   )$ is a finite map whose fibers admits  transitive $\mathbb{Z}/2\mathbb{Z}$-actions, where $ \Gr_{\rm sm}$ is the union of all $\Gr_\lambda$ such that $\lambda$ is small. Moreover, with respect to $\pi$,  Achar--Henderson \cite{AH,AHR} related the geometric Satake correspondence and Springer correspondence. 

In this paper, we consider a twisted analogue, and we will extend some results of Achar--Henderson in \cite{AH}. Let $\sigma$ be a diagram automorphism of order $2$, and let $\sigma$ act on the field $\mathcal{K}=\mathbb{C}((t))$ via $\sigma(t)=- t$ and $\sigma|_{\mathbb{C}}= {\rm Id}_{\mathbb{C}}$.  Then, we may define a twisted affine Grassmannian $\mathcal{G}r:=G(\mathcal{K})^\sigma/G(\mathcal{O})^\sigma$, where $\mathcal{O}=\mathbb{C}[[t]]$. Each twisted Schubert cell $\mathcal{G}r_{\bar{\lambda}}$, i.e. a $G(\mathcal{O})^\sigma$-orbit, is parametrized by the image $\bar{\lambda}$ of a dominant coweight $\lambda$ in the coinvariant lattice  $ X_*(T)_\sigma$ with respect to the induced action of $\sigma$, where $X_*(T)$ is the coweight lattice of $G$. In fact, $X_*(T)_\sigma$ can be regarded as the weight lattice of a reductive group $H:=(\check{G})^\sigma$, where $\check{G}$ is the Langlands dual group of $G$. 

 Let $\mathcal{G}r_0^-$ be the opposite open Schubert cell in $\mathcal{G}r$. We may naturally define a map $\pi: \mathcal{G}r_0^-\to  \mathfrak{p}$, where $\mathfrak{p}$ is the $(-1)$-eigenspace of $\sigma$ in $\mathfrak{g}$.  Let $\mathcal{M}_{\bar{\lambda}}$ denote the intersection $\mathcal{G}r_{\bar{\lambda}}\cap  \mathcal{G}r_0^- $, which is a nonempty open subset of $\mathcal{G}r_{\bar{\lambda}}$. 
 The following theorem is the main result of this paper, and it can follow from Proposition \ref{prop2.4} in Section \ref{subsection:Root datum} and Theorem \ref{theorem:cellandorbit} in Section \ref{section:Schubert_Nilpotent}, based on case-by-case analysis. 
\begin{theorem}
\label{mainthm}
Assume that $G$ is of type $A_{\ell} $ or $D_{\ell+1}$. The image $\pi( \mathcal{M}_\lambda )$ is contained in the nilpotent cone $\mathcal{N}_\mathfrak{p}$ of $\mathfrak{p}$, if and only if $\bar{\lambda}$ is a small dominant weight with respect to $H$. 
\end{theorem}
If we replace the field $\mathbb{C}$ by an algebraically closed field $\mathrm{k}$ of positive characteristic $p$, this theorem still holds for $p$ with minor restrictions, see Theorem \ref{main_thm_positve_char}.

In Theorem \ref{theorem:cellandorbit}  we describe precisely $\pi( \mathcal{M}_{\bar{\lambda}} ) $ as a union of nilpotent orbits in $\mathfrak{p}$ for each small $\bar{\lambda}$.   In Theorem \ref{theorem: isoA2l}, Theorem \ref{theorem: isoA2l-1}, and Theorem \ref{theorem:piD}, we also determine all small $\bar{\lambda}$ such that $\pi( \mathcal{M}_{\bar{\lambda}})$  is a nilpotent orbit and $\pi:  \mathcal{M}_{\bar{\lambda}}  \to   \pi( \mathcal{M}_{\bar{\lambda}} ) $ is an isomorphism. Furthermore, we describe all fibers of $\pi: \mathcal{M}  \to \pi(\mathcal{M})$ in Proposition \ref{prop:fiber1} and Proposition \ref{prop:fiber2}, where $\mathcal{M}$ is the union of $\mathcal{M}_{\bar{\lambda}}$ for all small $\bar{\lambda}$. The fibers are closely related to anti-commuting nilpotent varieties for symmetric spaces.  When $G$ is of type $A_{2\ell-1}$ (resp.\,$D_{\ell+1}$),  the reduced fiber $\pi^{-1}(0)_{\rm red}$ is actually the minimal (resp.\,maximal) order 2 nilpotent variety in $\mathfrak{sp}_{2\ell}$ (resp.\,$\mathfrak{so}_{2\ell+1} $). This is a very different phenomenon from the untwisted setting in the work of Achar--Henderson \cite{AH}, and it actually makes the twisted setting more challenging. 

For general simple Lie algebra $\mathfrak{g}$ and general diagram automorphism $\sigma$,  it was proved in \cite[Appendix C]{HLR} by Haines-Louren\c{c}o-Richarz that, when $\bar{\lambda}$  is quasi-miniscule and $\overline{\mathcal{O}}$ is the minimal nilpotent variety in $\mathfrak{p}$, the map $\pi: \overline{\mathcal{G}r}_{\bar{\lambda}}\cap \mathcal{G}r_0^-\to  \overline{\mathcal{O} }$ is an isomorphism.  In fact, we have also obtained this result independently, cf.\,\cite{Ko}.   Also, under the same assumption as in Theorem \ref{mainthm}, this isomorphism is a special case of our Theorem \ref{theorem: isoA2l}, Theorem \ref{theorem: isoA2l-1}, and Theorem \ref{theorem:piD}. 

The geometric Satake correspondence for $\mathcal{G}r$ was proved by Zhu \cite{Zh}, and it exactly recovers the Tannakian group $H$. On the other hand, the Springer correspondence for symmetric spaces is more sophisticated than the usual Lie algebra setting, see a survey on this subject \cite{Sh}.  It would be interesting to relate these two pictures as was done in \cite{AH,AHR}. Y.\,Li \cite{Li} defined the symmetric space analogue called $\sigma$-quiver variety in the setting of Nakajima quiver variety, and he showed that certain $\sigma$-quiver variety can be identified with null-cone of symmetric spaces. It is an interesting question to investigate a connection between $\sigma$-quiver variety and twisted affine Grassmannian in the spirit of the work of Mirkovi\'c-Vybornov \cite{MV}.

From Theorem \ref{mainthm}, we can deduce  some applications for the order 2 nilpotent varieties in classical symmetric spaces. Let $\langle ,\rangle $ be a symmeric or symplectic non-degenerate bilinear form on a vector space $V$. Let $\mathcal{A}$ be the space of self-adjoint linear maps with respect to $\langle ,\rangle $. We consider ${\rm Sp}_{2n}$-action on $\mathcal{A}$ when $\langle ,\rangle $ is symplectic and $\dim V=2n$, and ${\rm SO}_n$-action when $\langle ,\rangle $ is symmetric and $\dim V=n$.
  In Section \ref{sect:applications}, we obtain the following results. 
\begin{theorem}
\label{them_application}
\begin{enumerate}
\item If $\langle ,\rangle$ is symmetric and $\dim V$ is odd, then any order  2 nilpotent variety in $\mathcal{A}$ is normal.
\item If $\langle ,\rangle $ is symplectic, then there is a bijection of order 2 nilpotent varieties in $\mathfrak{so}_{2n+1}$ and in $\mathcal{A}$, such that they have the same cohomology of stalks of  IC-sheaves.  
\item  If $\langle ,\rangle $ is symplectic,  the smooth locus of any order 2 nilpotent variety in $\mathcal{A}$ is the open nilpotent orbit.
\end{enumerate}
\end{theorem}
It is known that when $\langle, \rangle$ is symplectic, any nilpotent variety in $\mathcal{A}$ is normal, but it is not always true when $\langle ,\rangle$ is symmetric, cf.\,\cite{Oh}. 
Using our methods, we can also prove that there is a bijection of order 2 nilpotent varieties in $\mathfrak{sp}_{2n}$ and in the space of symmetric $(2n+1)\times (2n+1)$ matrices, such that they have the same cohomlogy of stalks of  IC-sheaves. This was already proved earlier by Chen-Vilonen-Xue \cite{CVX} using different methods.  Also, Part 3) of Theorem \ref{them_application} is not true when $\langle ,\rangle$ is symmetric and $\dim V$ is odd, see more detailed discussions in Section  \ref{sect:applications}.
 \vspace{1em}

\noindent {\bf Acknowledgments}:  
This project grew out from a conversation with Yiqiang Li in March 2019. We would like to thank him for inspiring discussions. We also would like to thank the anonymous referee for very careful reading and many helpful comments and suggestions. 
J.\,Hong is partially supported by NSF grant DMS-2001365.

\section{Notation and Preliminaries}
\subsection{Root datum}
\label{subsection:Root datum}
	Let $G$ be a simply-connected simple algebraic group over $\mathbb{C}$, and let $\mathfrak{g}$ be its Lie algebra.  Let $\sigma$ be a diagram automorphism of $G$ of order $r$, preserving a maximal torus $T$ and a Borel subgroup $B$ containing $T$ in $G$.  Then $G$ has a root datum $(X_*(T),X^* (T), \langle\cdot ,\cdot \rangle,  \check{\alpha}_i,\alpha_i,i\in I)$ with the action of $\sigma$, where 
	\begin{itemize}
	\item $X_*(T)$ (resp. $X^*(T)$ ) is the coweight (resp. weight) lattice;
	\item $I$ is the set of vertices of the Dynkin diagram of $G$;
	\item $\alpha_i$ (resp. $\check{\alpha}_i $) is the simple root (resp. coroot) for each $i\in I$;
	\item $\langle\cdot ,\cdot \rangle: X_*(T)\times X^*(T)\to \mathbb{Z}$ is the perfect pairing. 
	\end{itemize}
	The automorphism $\sigma$ of this root datum satisfies
	\begin{itemize}
        \item $\sigma(\alpha_i)=\alpha_{\sigma(i)}$ and $\sigma(\check{\alpha}_i)=\check{\alpha}_{\sigma(i)}$;
		\item $\langle \sigma(\check{\lambda}),\sigma(\mu) \rangle=\langle \check{\lambda}, \mu\rangle$ for any $\check{\lambda}\in X_*(T)$ and $\mu\in X^* (T)$.
	\end{itemize} 
	As a diagram autormophism on $G$, $\sigma$ also preserves a pinning with respect to $B$ and $T$, i.e. there exists root subgroups $x_i, y_i$ associated to $\alpha_i, -\alpha_i$ for each $i\in $, such that
	\[  \sigma(x_i(a))=x_{\sigma(i)}(a), \quad  \sigma(y_i(a))=y_{\sigma(i)}(a),\quad \text{ for any } a\in \mathbb{C} .  \]
	


	
	Let $I_\sigma$ be the set of $\sigma$-orbits in $I$. Denote $X^*(T)^{\sigma}=\{\lambda\in X^*(T)\mid \sigma \lambda=\lambda\}$ and $X_* (T)_{\sigma}=X_*(T)/(\mathrm{Id} -\sigma)X_*(T)$. For each $\imath\in I_\sigma$, define $\gamma_\imath=\bar{\check{\alpha}}_i\in X_{*} (T)_{\sigma}$ for any $i\in \imath$, and define $\check{\gamma}_\imath\in X^{*}(T)^{\sigma}$ by
	\[ \check{\gamma}_\imath =
	\begin{cases}
	\sum_{i\in \imath}\alpha_i & \quad \text{if no pairs in } \imath \text{ is adjacent, }\\
	2\sum_{i\in \imath}\alpha_i & \quad  \text{if } \imath=\{i,\sigma(i)\} \text{ and } i \text{ and } \sigma(i)\text{ are adjacent,}\\
	\alpha_i  & \quad \text{if } \imath=\{i\}.
	\end{cases}
	\]
Let $\check{G}$ denote the Langlands dual group of $G$, and we still denote the induced diagram automorphism on $\check{G}$ by $\sigma$. 	Denoted by $H=(\check{G})^\sigma$ the $\sigma$-fixed subgroup of $\check{G}$. Then, $H$  has the root datum $(X^{*}(T)^{\sigma}, X_{*} (T)_{\sigma},\check{\gamma_\imath},\gamma_\imath,\imath\in I_\sigma)$, cf.\,\cite[Section 2.2]{HS}. For $\bar{\lambda},\bar{\mu}\in X_{*} (T)_{\sigma}$, define the partial order $\bar{\mu} \preceq \bar{\lambda}   $ if $\bar{\lambda}-\bar{\mu}$ is a sum of positive roots of $H$. Let $X_{*} (T)_{\sigma}^+$ be the set of dominant weight of $H$. In fact, $X_{*} (T)_{\sigma}^+$ is the image of the quotient map $X_{*} (T)^+ \to X_{*} (T)_{\sigma}$, where $X_{*} (T)^+$ is the set of dominant weights of $G$.

\subsection{Twisted affine Grassmannian}	
Let $\sigma$ be a diagram automorphism of $G$ of order $r$. Let $\mathcal{O}$ denote the set of formal power series in $t$ with coefficients in $\mathbb{C}$ and denote $\mathcal{K}$ the set of Laurent series in $t$ with coefficients in $\mathbb{C}$. Denote the automorphism $\sigma$ of order $r$ on $\mathcal{K}$ and $\mathcal{O}$ given by $\sigma$ acts trivially on $\mathbb{C}$ and maps $t\to \epsilon t$ where we fix the primitive $r$-root of unity $\epsilon$. 		We  consider the following {\it twisted affine Grassmannian }  attached to $G$ and $\sigma$,
	$$\mathcal{G}r_G=G(\mathcal{K})^{\sigma}/G(\mathcal{O})^{\sigma}.$$
This space has been studied intensively in \cite{BH,HR,PR,Ri}. The ramified geometric Satake correspondence \cite{Zh} asserts that there is an equivalence between the category of spherical perverse sheaves on $\mathcal{G}r_G$ and the category of representations of the algebraic group $H=(\check{G})^\sigma$. 	If there is no confusion, we write $\mathcal{G}r$ for convenience.

Let $e_0$ be the based point in $\mathcal{G}r$. For any $\lambda\in X_{*} (T)$,  we attach an element $t^{\lambda}\in T(\mathcal{K})$ naturally and define the norm $n^{\lambda}\in T(\mathcal{K})^\sigma$ of $t^{\lambda}$ by
\begin{equation}\label{equation:norm}
	n^{\lambda} :=\prod_{i=0}^{r-1}\sigma^i (t^{\lambda})=\epsilon^{\sum_{i=1}^{r-1}i \sigma^i(\lambda)}t^{\sum \sigma^i \lambda}.
\end{equation}
This construction originally occurred in \cite[Section 7.3]{Kot}. Let $\bar{\lambda}$ be the image of $\lambda$ in $X_*(T)_\sigma$.  Set $e_{\bar{\lambda}}=n^{\lambda}\cdot e_0\in \mathcal{G}r$.  Then  $e_{\bar{\lambda}}$ only depends on $\bar{\lambda}$. 
	Following \cite{BH,Zh}, $\mathcal{G}r$ admits the following Cartan decomposition
\begin{equation}\label{equation:Cartan}
	\mathcal{G}r=\bigsqcup_{\bar{\lambda}\in X_{*} (T)_{\sigma}^+}\mathcal{G}r_{\bar{\lambda}}
\end{equation}
	where $\mathcal{G}r_{\bar{\lambda}}=G(\mathcal{O})^\sigma\cdot e_{\bar{\lambda}}$ is a Schubert cell. Let $ \overline{  \mathcal{G}r}_{\bar{\lambda}}$ be the closure of $ \mathcal{G}r_{\bar{\lambda}}$.
	Then 
	\[  \overline{  \mathcal{G}r}_{\bar{\lambda}} =  \bigsqcup_{\bar{\mu} \preceq    \bar{\lambda} }\mathcal{G}r_{\bar{\mu}} ,\]
	and 
	$\dim \overline{\mathcal{G}r}_{\bar{\lambda}}=\langle2\rho ,\bar{\lambda} \rangle, $
	where $\rho$ is the half sum of all positive coroots of $H$. 

	 By abuse of notation,  we still use $\sigma$ to denote the induced automorphism on $\mathfrak{g}$ of order $r$. Then there is a grading on $\mathfrak{g}$,
	 $$\mathfrak{g}=\mathfrak{g} _0\oplus \mathfrak{g}_1 \oplus \cdots \oplus \mathfrak{g}_{r-1}$$
	 where $\mathfrak{g}_i$ is the $\epsilon^i$-eigenspace.   Set
	  \[ \mathfrak{p}= \mathfrak{g}_1.\] 
	 
	  Set $\mathcal{O}^-=\mathbb{C}[t^{-1}]$.  Consider the evaluation map ${\rm ev}_\infty: G(\mathcal{O}^-)\to G$. Let $G(\mathcal{O}^-)_0$ denote its kernel.	 The map ${\rm ev}_\infty$  factors through $G(\mathbb{C}[t^{-1}]/(t^{-2}))\to G$. 
	 Note that the kernel of $G(\mathbb{C}[t^{-1}]/(t^{-2}))\to G$ is canonically identified with the vector space $\mathfrak{g} \otimes t^{-1}$ with respect to the adjoint action of $G$ and $\sigma$.  It induces a $G\rtimes \langle \sigma \rangle$-equivariant map 
	 \[ G(\mathcal{O}^-)_0\to \mathfrak{g}\otimes t^{-1}. \]
	 Taking $\sigma$-invariants, we get a $K$-equivariant map 	 \begin{equation}  
	 \label{equation:invariant}
	  G(\mathcal{O}^-)_0^\sigma\to \mathfrak{p}  ,\end{equation}
	where $K:= G^\sigma$. Note that $K$ is a connected simply-connected simple algebraic group, as $G$ is simply-connected.

	Set  $ \mathcal{G}r_0^- := G(\mathcal{O}^-)^\sigma\cdot e_0\simeq G(\mathcal{O}^-)_0^\sigma $.
	Then   $\mathcal{G}r_0^-$ is the open opposite Schubert cell in $\mathcal{G}r$.  From (\ref{equation:invariant}), we have the following $K$-equivariant map
	 \begin{equation}  
	 \label{equation:pi}
	\pi:  \mathcal{G}r_0^-\to \mathfrak{p}  . \end{equation}
	\begin{lemma}
	$\mathcal{G}r_{\bar{\lambda}} \cap\mathcal{G}r_0^-  $ is nonempty for any $\lambda\in X_{*} (T)_{\sigma}^+$.
\end{lemma}
\begin{proof}
First note that $\mathcal{G}r_0^-$ is an open subset in $\mathcal{G}r$, cf.\,\cite{BH}[Proof of Theorem 4.2].  Moreover, $\mathcal{G}r_0^-\cap  \overline{\mathcal{G}r}_{\overline{\lambda}}$ contains the base point $e_{0}$. Thus, the intersection $\mathcal{G}r_0^-\cap  \overline{\mathcal{G}r}_{\overline{\lambda}}$ is a nonempty open subset of $ \overline{\mathcal{G}r}_{\overline{\lambda}}$. Hence $\mathcal{G}r_{\bar{\lambda}} \cap\mathcal{G}r_0^-  $  is also nonempty.


\end{proof}
	 

	Following \cite{Br,Re, AH},   an element $\bar{\lambda}$ of $X_* (T)_{\sigma}^+$ is called {\it small},  if $\bar{\lambda}\nsucceq 2\gamma_0$, where $\gamma_0$ is the highest short root of $H$. The set of all small dominant weights is a lower order ideal of $X_{*} (T)_{\sigma}^+$, i.e., if $\bar{\mu}\preceq\bar{\lambda}$ and $\bar{\lambda}$ is small, then $\bar{\mu}$ is also small. Let $\mathcal{G}r_{\mathrm{sm}}$ be the union of $\mathcal{G}r_{\bar{\lambda}}$ for small dominant weights $\bar{\lambda}$. 
	 Set
	 $$\mathcal{M}=\mathcal{G}r_{\mathrm{sm}}\cap \mathcal{G}r_0^-.$$
	 For each small dominant weight $\bar{\lambda}$, set
	 $$\mathcal{M}_{\bar{\lambda}}=\mathcal{G}r_{\bar{\lambda}}\cap \mathcal{G}r_0^-.$$
Let  $\mathcal{N}_{\mathfrak{p}}$ denote the nilpotent cone of $\mathfrak{p}$. 
We shall prove in Section  \ref{section:Schubert_Nilpotent} that $\pi(\mathcal{M})$ is contained $\mathcal{N}_{\mathfrak{p}}$, when $G$ is of type $A_n$ and $D_{n}$ and $\sigma$ is of order $2$.

Recall that $\gamma_0$ is the highest short root of $H$. The following lemma is a twisted analogue of \cite[Lemma 3.3]{AH}. 
\begin{lemma}\label{lemma:highest}
	If $\sigma$ is a diagram automorphism of order $r$, then $\pi(\mathcal{G}r_{2\gamma_0}\cap \mathcal{G}r_0^-)\nsubseteq \mathcal{N}_{\mathfrak{g}_1}$.
\end{lemma}
\begin{proof}
Let $X_N$ be the Dynkin diagram of $G$. Following \cite[p.128-129]{Ka}, we choose the following root of $G$,
\[ {\theta}_0=\begin{cases}       
 {\alpha}_1+\cdots + {\alpha}_{2\ell-2}, &\quad   (X_N, r)=(A_{2\ell-1}, 2);\\
 {\alpha}_1+\cdots + {\alpha}_{2\ell},  &\quad   (X_N, r)=(A_{2\ell}, 2);\\
 {\alpha}_1+\cdots + {\alpha}_{\ell},  &\quad   (X_N, r)=(D_{\ell+1}, 2);\\
 {\alpha}_1+{\alpha}_2+{\alpha}_3,  &\quad   (X_N, r)=(D_{4}, 3);\\
 {\alpha}_1+2{\alpha}_2+2{\alpha}_3+{\alpha}_4+{\alpha}_5+{\alpha}_6,  &\quad   (X_N, r)=(E_{6}, 2).
\end{cases} \]
where the label of simple roots ${\alpha}_i$ follows from \cite[TABLE Fin, p.53]{Ka}. Recall from the section 2.1 that for each $\iota\in I_\sigma$, we define simple roots of $H$, $\gamma_\imath=\bar{\check{\alpha}}_i\in X_{*} (T)_{\sigma}$. 

Let $\check{\theta}_0$ be the coroot of ${\theta}_0$. Then,
\[ \bar{\check{\theta}}_0=\begin{cases}
\gamma_0    &\quad   \text{if } (X_N,r)\not= (A_{2\ell}, 2)\\
2\gamma_0  &\quad    \text{if } (X_N,r)= (A_{2\ell}, 2)
\end{cases} .  \]

	Suppose $(X_N,r)\not= (A_{2\ell}, 2)$.  Note that $\theta_0\in X^*(T)$ and $\check{\theta}_0:\mathbb{C}^\times \to T$. Each $a\in\mathbb{C}^\times$ can be identified with $\begin{pmatrix}
	a & 0\\
	0 & a^{-1}
	\end{pmatrix}\in \mathrm{SL}_2.$ For each $i=0,...,r-1$, define a homomorphism $\phi_{\sigma^i(\theta_0)}:\mathrm{SL}_2 \to G$ given by
	\[\begin{pmatrix}
	1 & a\\
	0 & 1
	\end{pmatrix}\mapsto x_{\sigma^i(\theta_0)}(a), \hspace{0.5 cm}
	\begin{pmatrix}
	1 & 0\\
	a & 1
	\end{pmatrix}\mapsto y_{\sigma^i(\theta_0)}(a), \hspace{0.5 cm}
	\begin{pmatrix}
	a & 0\\
	0 & a^{-1}
	\end{pmatrix}\mapsto \sigma^i(\check{\theta}_0)(a).
	\]
	Let $\mathcal{S}$ be the product of $r$ copies of $\mathrm{SL}_2$. Then $\check{\theta}_0$ can be extended to $\phi:\mathcal{S} \to G$ given by
	$$\phi(g_0,...,g_{r-1})=\prod_{i=0}^{r-1}\phi_{\sigma^i(\theta_0)}(g_i).$$
	This $\phi$ can extend scalar to $\mathcal{K}$.
	Abusing notation, define $\sigma:\prod_{i=1}^{r}(\SL_2(\mathbb{\mathcal{K}}))_i \to \prod_{i=1}^{r}(\SL_2(\mathbb{\mathcal{K}}))_i$ by
	$$\sigma(g_1(t),g_2(t),...,g_r(t))=(g_r(\epsilon t),g_1(\epsilon t),...,g_{r-1}(\epsilon t)).$$ 
	There exists an isomorphism 
	\[  \varphi:\mathrm{SL}_2(\mathcal{K})\to(\prod_{i=1}^{r}(\SL_2(\mathbb{\mathcal{K}}))_i)^\sigma=\{(g(t),g(\epsilon t),...,g(\epsilon^{r-1} t))\mid g(t)\in \SL_2(\mathcal{K})\}.\]
	Hence
	\[\phi\circ \varphi:\begin{pmatrix}
	t & 0\\
	0 & t^{-1}
	\end{pmatrix}\mapsto \left(
	\begin{pmatrix}
	\epsilon^i t & 0\\
	0 & (\epsilon^i t)^{-1}
	\end{pmatrix}\right)_{i=0,...,r-1}
	\mapsto \prod_{i=0}^{r-1}(\epsilon^i t)^{\sigma^i \check{\theta}_0 }=n^{\check{\theta}_0}.
	\]
	
	Let $\mathfrak{s}$ be the product of $r$ copies of $\mathfrak{sl}_2$. Define $\sigma: \mathfrak{s}\to \mathfrak{s}$ by
	$$\sigma(x_1,,...,x_{r-1},x_r)=(\epsilon x_r,\epsilon x_1,...,\epsilon x_{r-1}).$$
	Since $\sigma$ has order $r$, we have $\mathfrak{s}=\oplus_{i=0}^{r-1}\mathfrak{s}_i$ where $\mathfrak{s}_i$ is the eigenspace of eigenvalue $\epsilon^i$. Then $\mathfrak{s}_1=\{(x,\epsilon x,...,\epsilon^{r-1} x)\mid x\in \mathfrak{sl}_2\}\cong \mathfrak{sl}_2$. The derivative of $\phi$ is $\mathrm{d}\phi:\mathfrak{s}\to \mathfrak{g}$ which induces $\mathfrak{s}_1\to\mathfrak{g}_1$. Hence we have the map $\Psi:\mathfrak{sl}_2\to \mathfrak{g}_1$.

	Consider the matrix $g(t)\in \SL_2(\mathcal{O}^-)$,
	\[g(t)=\begin{pmatrix}
	1+t^{-1} & t^{-2}\\
	t^{-1} & 1-t^{-1}+t^{-2}
	\end{pmatrix}=
	\begin{pmatrix}
	0 & 1\\
	-1 &t^2 -t+1
	\end{pmatrix}
	\begin{pmatrix}
	t^2 & 0\\
	0 & t^{-2}
	\end{pmatrix}
	\begin{pmatrix}
	1 & 0\\
	t^2 +t & 1
	\end{pmatrix}.
	\]
	Then $(\phi\circ \varphi)(g(t))\in G(\mathcal{O})^\sigma n^{2 \check{\theta}_0} G(\mathcal{O})^\sigma$. Since $G$ is not type $A_{2l}$, $\bar{\check{\theta}}_0=\gamma_0$ and then $(\phi\circ \varphi)(g(t))\cdot e_0\in \mathcal{G}r_{2\gamma_{0}}\cap \mathcal{G}r_{G,0}^-$. We have the commutative diagram
	\[\begin{tikzcd}
	{\Gr_{\mathrm{SL}_2,0}^-} &&&& {\mathcal{G}r_{\mathcal{S},0}^-} &&&& {\mathcal{G}r_{G,0}^-} \\
	\\
	{\mathfrak{sl}_2} &&&& {\mathfrak{s}_1} &&&& {\mathfrak{g}_1}
	\arrow["{g(t)\cdot L_0 \mapsto \varphi(g(t))\cdot e_0}", from=1-1, to=1-5]
	\arrow["{(g_i(t))_{i=0}^{r-1}\cdot e_0 \mapsto \phi(g_i(t))_{i=0}^{r-1})\cdot e_0}", from=1-5, to=1-9]
	\arrow["{\pi_{\mathrm{SL}_2}}"', from=1-1, to=3-1]
	\arrow[from=1-5, to=3-5]
	\arrow["{\pi}", from=1-9, to=3-9]
	\arrow["{x\mapsto (x,\epsilon x,...,\epsilon^{r-1}x)}"', from=3-1, to=3-5]
	\arrow["{}"', from=3-5, to=3-9]
	\end{tikzcd}\]
	where $ \Gr_{\mathrm{SL}_2,0}^-:= \mathrm{SL}_2(\mathcal{O}^-)_0\cdot e_0\subset \Gr_{\mathrm{SL}_2}$, and ${\mathcal{G}r_{\mathcal{S},0}^-}$ is defined similarly. 
	The commutativity follows from
	$$\pi((\phi\circ \varphi)(g(t))\cdot e_0)=\Psi(\pi_{\SL_2}(g(t)\cdot e_0))=\Psi\begin{pmatrix}
	1 & 0\\
	1 & -1
	\end{pmatrix}  , $$
	where the latter is not nilpotent.  It follows that, $ \pi( \mathcal{G}r_{2 \gamma_0 }  )\not \subseteq   \mathcal{N}_{\mathfrak{p}  } $.
	
	Suppose that $(X_N,r)= (A_{2n}, 2)$. In this case, $\bar{\check{\theta}}_0=2\gamma_0$ and $\sigma(\check{\theta}_0)={\check{\theta}_0}$. Then $\check{\theta}_0$ can be extended to $\phi:\mathrm{SL}_2\to G$ defined by
	\[\begin{pmatrix}
	1 & a\\
	0 & 1
	\end{pmatrix}\mapsto x_{\theta_0}(a), \hspace{0.5 cm}
	\begin{pmatrix}
	1 & 0\\
	a & 1
	\end{pmatrix}\mapsto y_{\theta_0}(a), \hspace{0.5 cm}
	\begin{pmatrix}
	a & 0\\
	0 & a^{-1}
	\end{pmatrix}\mapsto \check{\theta}_0(a).
	\]
	$\phi$ can extend the scalar to $\mathcal{K}$.  	
	Define a group homomorphism $\sigma:\mathrm{SL}_2(\mathcal{K})\to \mathrm{SL}_2(\mathcal{K})$ by
	
	\[\begin{pmatrix}
	a(t) & b(t)\\
	c(t) & d(t)
	\end{pmatrix}\mapsto
	\begin{pmatrix}
	a(-t) & -b(-t)\\
	-c(-t) & d(-t)
\end{pmatrix}
	\]
	where $a(t)\in \mathcal{K}$. Then $\phi:\mathrm{SL}_2(\mathcal{K})\to G(\mathcal{K})$ is $\sigma$-equivariant. The induced homomorphism $\sigma:\mathfrak{sl}_2\to \mathfrak{sl}_2$ is given by
	\[\begin{pmatrix}
	a & b\\
	c & -a
	\end{pmatrix}\mapsto
	\begin{pmatrix}
	a & -b\\
	-c & -a
	\end{pmatrix}.
	\]
	The derivative $\mathrm{d}\phi:\mathfrak{sl}_2\to \mathfrak{g}$ induces the map $\Psi:(\mathfrak{sl}_2)_1\to \mathfrak{g}_1$. Similar to the above arguement, we have the commutative diagram
	\[\begin{tikzcd}
	{\mathcal{G}r_{\mathrm{SL}_2,0}^-} && {\mathcal{G}r_{\mathrm{G},0}^-} \\
	\\
	{(\mathfrak{sl}_2)_1} && {\mathfrak{g}_1}
	\arrow[from=1-1, to=1-3]
	\arrow["{\pi_{\mathrm{SL}_2}}"', from=1-1, to=3-1]
	\arrow["{\Psi}", from=3-1, to=3-3]
	\arrow["{\pi}", from=1-3, to=3-3]
	\end{tikzcd}\]
	where ($\mathfrak{sl}_2)_1$ is the eigenspace of eigenvalue $-1$ under $\sigma$. Now consider $g(t)\in \mathrm{SL}_2(\mathcal{O}^-)^\sigma$
		\[g(t)=\begin{pmatrix}
	1 & t^{-1}\\
	t^{-1} & 1+t^{-2}
	\end{pmatrix}=
	\begin{pmatrix}
	1 & -t^3+t\\
	0& 1
	\end{pmatrix}
	\begin{pmatrix}
	t^2 & 0\\
	0 & t^{-2}
	\end{pmatrix}
	\begin{pmatrix}
	1 & t\\
	t & 1+t^2
	\end{pmatrix}.
	\]
		Then $\phi(g(t))\in G(\mathcal{O})^\sigma n^{\check{\theta}_0} G(\mathcal{O})^\sigma$ and $\phi(g(t))\cdot e_0\in \mathcal{G}r_{2\gamma_0}\cap \mathcal{G}r_0^-$. The result follows from
	$$\pi(\phi(g(t))\cdot e_0)=\Psi(\pi_{\SL_2}(g(t)\cdot e_0))=\Psi\begin{pmatrix}
	0 & 1\\
	1 & 0
	\end{pmatrix}$$
	where the latter is not nilpotent. It also follows that, $ \pi( \mathcal{G}r_{2 \gamma_0 }  )\not \subseteq   \mathcal{N}_{\mathfrak{p}  } $.
\end{proof}	

\begin{prop}
\label{prop2.4}
	For $\bar{\lambda}\in X_{*} (T)_{\sigma}^+$, if $\pi(\mathcal{G}r_{\bar{\lambda}}\cap \mathcal{G}r_0^-)\subset \mathcal{N}_{\mathfrak{p}}$, then $\bar{\lambda}$ is small.
\end{prop}
\begin{proof}
	Since $\mathcal{G}r_0^-$ is an open subset of $\mathcal{G}r$, $\pi(\overline{\mathcal{G}r}_{\bar{\lambda}}\cap \mathcal{G}r_0^-)\subset \mathcal{N}_{\mathfrak{p}}$. By Lemma \ref{lemma:highest}, $\mathcal{G}r_{2\gamma_0} \nsubseteq  \overline{\mathcal{G}r}_{\bar{\lambda}}  $ which means $\bar{\lambda}\nsucceq 2\gamma_0$.
\end{proof}
	
	
	Define the following anti-involution
	$$\iota: G(\mathcal{K}) \to G(\mathcal{K}),\hspace{0.5 cm}g(t)\mapsto g(-t)^{-1}.$$
	It can be checked that $\iota$ commutes with $\sigma$, and $\iota$ preserves $G(\mathcal{K})^{\sigma},G(\mathcal{O})^{\sigma}$ and $K^-$. This induces the map
	$$\iota: \mathcal{G}r_0^- \to \mathcal{G}r_0^-, \hspace{0.5 cm} g(t)\cdot e_0\mapsto g(-t)^{-1}\cdot e_0.$$
The following lemma will be used in Section \ref{section:Schubert_Nilpotent}. 	
\begin{lemma}\label{lemma:iota}
	For $\bar{\lambda}\in X_{*} (T)_{\sigma}^+$, $\iota(\mathcal{M}_{\bar{\lambda}})\subset \mathcal{M}_{\bar{\lambda}}$.
\end{lemma}
	
		
	\begin{proof}
		It suffices to prove $\iota(n^\lambda)\in\mathcal{M}_{\bar{\lambda}}$ for each $\lambda\in X_{*} (T)^+$. 
		\begin{align*}
		\iota(n^{\lambda})&=\iota(\epsilon^{\sigma\lambda+2\sigma^2\lambda+...+(r-1)\sigma^{r-1}\lambda}t^{\sum_{i=0}^{r-1}\sigma^i\lambda})\\
		&=\epsilon^{-(\sigma\lambda+2\sigma^2\lambda+...+(r-1)\sigma^{r-1}\lambda)}(-1)^{\sum_{i=0}^{r-1}\sigma^i\lambda}t^{-\sum_{i=0}^{r-1}\sigma^i\lambda}\\
		&=(-1)^{\sum_{i=0}^{r-1}\sigma^i\lambda}n^{-\lambda}.
		\end{align*}
		Since $(-1)^{\sum_{i=0}^{r-1}\sigma^i\lambda}$ is fixed by $\sigma$, $\iota(n^{\lambda})\in G(\mathcal{O})^\sigma n^{-\lambda}G(\mathcal{O})^\sigma$.
		Let $W$ be the Weyl group of $G$ with respect to the maximal torus $T$ and $\omega_0$ the longest element of $W$. We can choose a representative $\dot{\omega}_0 \in G$ of $\omega_0$ such that $\sigma(\dot{\omega}_0 )=\dot{\omega}_0$, cf.\,\cite[Section 2.3]{HS}.  
		
		When $G$ is of type $D_{2\ell}$ with $\ell \geq 2$, $w_0=-1$; otherwise, $w_0=-\sigma$ and $\sigma$ is of order $2$, cf.\,\cite[Ex 5, p.71]{Hu2}. If $w_0=-1$, it is easy to see that $n^{-\lambda}=w_0n^{\lambda}w_0^{-1}$. If $w_0=-\sigma$ and $\sigma$ has order 2,
		$$n^{-\lambda}=(-1)^{-\sigma\lambda}t^{-(\lambda+\sigma\lambda)}=(-1)^{w_0\lambda}t^{w_0(\lambda+\sigma\lambda)}=w_0(-1)^\lambda t^{\lambda+\sigma\lambda}w_0^{-1}=w_0 (-1)^{\lambda+\sigma\lambda}n^{\lambda}w_0^{-1}.$$
		In any case, $\iota(n^\lambda)\in G(\mathcal{O})^\sigma n^{\lambda}G(\mathcal{O})^\sigma$.
	\end{proof}

\section{Nilpotent orbits in the space of self-adjoint maps}
\label{sect_3}
In this section, we will review some facts on the nilpotent orbits in certain symmetric spaces. These results are known, cf.\,\cite{Se}. We provide proofs here, as the proofs in \cite{Se} are omitted.

	Let $B=\langle \cdot , \cdot \rangle$ be a nondegenerate symmetric or skew-symmetric bilinear form on a vector space $V=\mathbb{C}^m$ and $\mathcal{A}$ the set of self-adjoint linear maps under the bilinear form.  In this section, we describe the classification of nilpotent orbits in the space $\mathcal{A}$ in Theorem \ref{theorem:classify}, Theorem \ref{theorem:SO_{2n+1}} and Theorem \ref{theorem:SO_{2n}}. 
		
	The isometry group of the form $B$ is
	$$I_B=\{g\in \mathrm{GL}(V)\mid \langle gu , gv \rangle=\langle u , v \rangle \text{ for all } u,v\in V\},$$
	whose Lie alegbra is
\begin{equation}  
	\mathfrak{g}_B := \{X\in \mathfrak{sl}(V)\mid \langle Xu , v \rangle+\langle u , Xv\rangle=0 \text{ for all } u,v\in V\}.
\end{equation}
When $B$ is symplectic, $\dim V$ is even, $I_B\cong \mathrm{Sp}_{2n}$ and $\mathfrak{g}_B\simeq \mathfrak{sp}_{2n}$ where $m=2n$. When $B$ is symmetric, $I_B\cong \mathrm{O}_m$ and $\mathfrak{g}_B\cong \mathfrak{so}_m$.

	
	The group $I_B$ acts on the space of self-adjoint linear maps
\begin{equation}\label{equation:spaceA}
	\mathcal{A}=\{X\in \End(V)\mid \langle Xu , v \rangle=\langle u , Xv \rangle \text{ for all } u,v\in V\}
\end{equation}
	by conjugation. The orbit is called nilpotent if it is the orbit of a nilpotent element of $\mathcal{A}$. 
	
	Let $\mathfrak{g}$ be a complex semisimple Lie algebra. Suppose that $\mathfrak{g}$ has $\mathbb{Z}_m$-grading
	$$\mathfrak{g}=\bigoplus_{i\in \mathbb{Z}_m }\mathfrak{g}_i$$
	so that $[\mathfrak{g}_k,\mathfrak{g}_\ell]\subset \mathfrak{g}_{k+\ell}$. We have the following graded version of Jacobson--Morozov Theorem and Kostant Theorem.



\begin{lemma}\label{lemma:JMgrading}
	Let $X$ be a nonzero nilpotent element in $\mathfrak{g}_{i}$. 
	\begin{enumerate}[1)]
		\item There exists an $\mathfrak{sl}_2$-triple $\{H,X,Y\}$ such that $H\in \mathfrak{g}_{0}$ and $Y\in \mathfrak{g}_{-i}$.
		\item Let $\{H',X,Y'\}$ be another $\mathfrak{sl}_2$-triple such that $Y'\in\mathfrak{g}_{-i}$ and $H'\in\mathfrak{g}_0$. Then there exists $g\in K^X$ such that $g\cdot H=H'$, $g\cdot X=X$ and $g\cdot Y=Y'$.
	\end{enumerate}
	
\end{lemma}
\begin{proof}
The first part follows from the usual Jacobson--Morozov Theorem, and the proof is similar to \cite[Lemma 1.1]{EK}. We replace $\mathfrak{u}^X:=\mathfrak{g}^X\cap [\mathfrak{g},X]$ in \cite[Lemma 3.4.5]{CM} by $\mathfrak{u}_0^X:=\mathfrak{g}_0^X\cap [\mathfrak{g},X]$. Then the proof of the second part is similar to \cite[Theorem 3.4.10]{CM}.
\end{proof}

For each $A\in \mathfrak{sl}(V)$, as a linear map, we denote its adjoint by $A^*$ under the form $B$. We define an involution $\sigma$ on $\mathfrak{g}=\mathfrak{sl}(V)$ by
	\begin{equation}\label{equation:sigma}
	\sigma(A)=-A^*.
	\end{equation}
	Then $\mathfrak{g}$ is the direct sum of eigenspaces, $\mathfrak{g}=\mathfrak{g}_0\oplus \mathfrak{g}_1$. Thus $\mathfrak{g}_0=\mathfrak{g}_B$ and $\mathfrak{g}_1=\mathcal{A}$.  Fix a nonzero nilpotent element $X\in \mathcal{A}$.  By Lemma \ref{lemma:JMgrading}, there exists $Y\in \mathcal{A}$ and $H\in \mathfrak{g}_B$, such that $X,Y,H$ is an $\mathfrak{sl}_2$-triple. 
This induces a representation of $\mathfrak{sl}_2$ on $V$ and hence we have a decomposition
	\begin{equation}\label{eq:decomp}
	V=\bigoplus_{r\geq 0}M(r)
	\end{equation}
	where $M(r)$ is a finite direct sum of irreducible representation of $\mathfrak{sl}_2$ of highest weight $r$.
	For $r\geq 0$, let $H(r)$ be the highest weight space in $M(r)$. Define a new bilinear form $(\cdot,\cdot)$ on $H(r)$ by
	$$(u,v)_r=\langle u,Y^r v\rangle.$$
	
\begin{lemma}\label{lemma:newform}
	For any $r\geq 0$, $(\cdot,\cdot)_r$ is symplectic (resp. symmetric) if $B$ is symplectic (resp. symmetric). 
\end{lemma}
\begin{proof}
We assume $B$ is symplectic.  The proof is similar when $B$ is symmetric. 
	It is easy to see that $(\cdot,\cdot)_r$ is skew-symmetric. It remains to show that $(\cdot,\cdot)_r$ is nondegerate. Let $V_r$ be an $r$-weight space in $\mathbb{C}^{2n}$. For any $u\in V_r$, $v\in V_s$ with $s\neq -r$,
	$$(r+s)\langle u,v \rangle = \langle ru,v \rangle + \langle u,sv \rangle=\langle Hu,v \rangle+\langle u,Hv \rangle=0$$
	This implies that $V_r$ and $V_s$ are $\langle \cdot,\cdot \rangle$-orthogonal. Let
	$$W=\text{Span}\{u\in V_r \mid u=Yv \text{ for some } v\in \mathbb{C}^{2n}\}.$$
	It can be seen that $V_r=H(r)\oplus W$. For $u\in H(r)$ and $v\in W$, write $v=Yv'$,
	$$(u,v)_r=\langle u,Y^r v\rangle=\langle u,Y^{r+1}v'\rangle=\langle Y^{r+1}u,v'\rangle=0.$$
	Hence $H(r)$ is $(\cdot,\cdot)_r$-orthogonal to $W$.
	
	
	We claim that $\langle \cdot,\cdot \rangle:(Y^r\cdot H(r))\times H(r)\to \mathbb{C}$ is nondegenerate. Let $u=Y^r u'\in Y^r\cdot H(r)$ be such that $\langle u,v\rangle=0$ for all $v\in H(r)$. For each $w\in \mathbb{C}^{2n}$, write $w=\sum_{s} w_s$ where each $w_s$ belongs to $V_s$. Since $u\in V_{-r}$, $\langle u,w_s\rangle=0$ for $s\neq r$. Write $w_r=w_1+w_2$ where $w_1\in H(r)$ and $w_2=Y w'_2\in W$. By the assumption $\langle u,w_1\rangle=0$ and hence
	$$\langle u,w_r\rangle=\langle u,w_2\rangle=\langle Y^r u',Y w'_2\rangle=\-\langle Y^{r+1} u', w'_2\rangle=0.$$
	We obtain $\langle u,w\rangle=0$ for any $w$ and hence $u=0$. This claim implies that $(\cdot,\cdot)_r$ is nondegenerate.
\end{proof}	

	A partition of a positive integer is denoted by a tuple $[d_1,d_2,...,d_k]$ of positive integers. We use the exponent notation 
	$$[a_{1}^{i_1},...,a_{r}^{i_r}]$$
	to denote a partition where $a_{j}^{i_j}$ means there are $i_j$ copies of $a_{j}$. For example, $[3^2,1^4]=[3,3,1,1,1,1]$ is a partition of 10. Put $r_i=|\{j\mid d_j=i\}|$ and $s_i=|\{j\mid d_j\geq i\}|$. In fact, each partition can be illustrated by Young diagram and then $s_i$ is the $i$-th part of the dual diagram. The following Theorem gives the parametrization of nilpotent $I_B$-orbits in $\mathcal{A}$.

\begin{theorem}\label{theorem:classify}
	There exists one-to-one correspondences
	$$\{\text{nilpotent } \mathrm{Sp}_{2n}\text{-orbits in } \mathcal{A}\}\leftrightarrow
	\left\{\begin{array}{c}
	\text{partitions of } 2n \text{ such that}	\\
	\text{every part occurs with even multiplicity}
	\end{array}\right\}.$$
	and
	$$\{\text{nilpotent } \mathrm{O}_m\text{-orbits in } \mathcal{A}\}\leftrightarrow
	\{\text{partitions of } m \}.$$		
\end{theorem}
\begin{proof}
	The proof is similar to \cite[Lemma 5.1.17]{CM}. For the case that $B$ is symplectic, it suffices to show that any nilpotent element in $\mathcal{A}$ gives rise to a partition of $2n$ such that every part occurs with even multiplicity. Given nilponent $X\in \mathcal{A}$, the number of Jordan blocks of size $r+1$ equals to the multiplicity of $M(r)$ in $\mathbb{C}^{2n}$ which is exactly $\dim H(r)$. By Lemma \ref{lemma:newform}, $\dim H(r)$ is even for every $r$.
	
	If $B$ is symmetric, there are no constraints on $\dim H(r)$ which means there are no conditions on partitions of $m$.
\end{proof}

	
	
\begin{theorem}\label{theorem:SO_{2n+1}}
	There exists one-to-one correspondence
	$$\{\text{nilpotent } \mathrm{SO}_{2n+1}\text{-orbits in } \mathcal{A}\}\leftrightarrow
	\left\{\text{partitions of } 2n+1 \right\}.$$
\end{theorem}	
\begin{proof}
	Since $\mathrm{O}_{2n+1}=\mathrm{SO}_{2n+1}\times \{\pm I_{2n+1}\}$, the orbits under $\mathrm{O}_{2n+1}$ and $\mathrm{SO}_{2n+1}$ coincide. The results immediately follows from Theorem \ref{theorem:classify}
\end{proof}

	Consider the case that $B$ is symmetric and $m=2n$. Given nilpotent elements $X,X'\in\mathcal{A}$ whose partitions are the same and have at least one odd part. Say that they are conjugated by an element $g\in \mathrm{O}_{2n}$. If $\det g=1$, we conclude that $X,X'$ are in the same $\mathrm{SO}_{2n}$-orbits. Suppose that $\det g=-1$. We modify this $g$ so that it has determinant $1$. By Lemma \ref{lemma:JMgrading}, $X$ gives rise to the decomposition (\ref{eq:decomp}). An odd part in the partition corresponds to an odd dimensional irreducible representation $S$ of $\mathfrak{sl}_2$ in $\mathbb{C}^{2n}$. We put $h=g$ except that $h(v)=-g(v)$ for $v\in S$. Therefore, $\det h=1$, and $X$ and $X'$ are conjugated by $h$. If there is no odd parts, we need the following Lemma.
	
\begin{lemma}\label{lemma:stabilizer}
	Let $X$ be a nilpotent element in $\mathcal{A}$ whose partition contains only even parts, and $k\in \mathrm{O}_{2n}$ such that $k\cdot X=kXk^{-1}=X$. Then $\det k=1$.
\end{lemma}
\begin{proof}
Let $\mathrm{O}_{2n}^X $ be the stabilizer group of $O_{2n}$ at $X$. Then $k\in \mathrm{O}_{2n}^X$. 	By multiplicative Jordan decomposition, cf.\,\cite[Theorem 4.4, p.83]{Bo}, let $k_s\in O_{2n}^X$ be the semisimple part of $k$. Then  $\det k_s=\det k$.  Hence we may assume that $k$ is semisimple. Let $\sigma$ be an automorphism on $\mathfrak{g}=\mathfrak{sl}(V)$ defined by (\ref{equation:sigma}). Then $\sigma$ commutes with ${\rm Ad} k$ on $\mathfrak{g}$, as $k\in \mathrm{O}_{2n}$. Then $\mathfrak{g}=\mathfrak{g}_0\oplus \mathfrak{g}_1$, and we have the decomposition of $k$-stabilizers $\mathfrak{g}^k=\mathfrak{g}_0^k\oplus \mathfrak{g}_1^k$ where $\mathfrak{g}_i^k=\mathfrak{g}_i\cap \mathfrak{g}^k$. Since $\mathfrak{g}^k$ is reductive and $X\in \mathfrak{g}_1^k$, by Lemma \ref{lemma:JMgrading}, there exists an $\mathfrak{sl}_2$-triple $H,X,Y$ such that $X,Y\in \mathfrak{g}_1^k$ and $H\in \mathfrak{g}_0^k$ and hence we have the decomposition (\ref{eq:decomp}). It is easy to see that $k(M(r))\subset M(r)$, and also, $k$ stabilizes each weight space of $M(r)$. 
	
	Recall that $\langle \cdot , \cdot \rangle$ is a nondegenerate symmetric form on $V=\mathbb{C}^{2n}$ and a form on $H(r)$ given by $(u,v)_r=\langle u , Y^r v \rangle$ is also symmetric for any $r\geq 0$. For any $u,v\in H(r)$,
	$$(ku,kv)_r=\langle ku, Y^r kv \rangle=\langle ku , kY^r v \rangle=\langle u , Y^r v\rangle=(u,v)_r.$$ Hence $k\big|_{H(r)}\in \mathrm{O}(H(r))$ for any $r$. In particular $\det\left( k\big|_{H(r)}\right)=\pm 1$.
	
	Let $M(r)_\ell$ be an $\ell$-weight space, and $L(r)$ the lowest weight space in $M(r)$. Observe that $X\big|_{M(r)_\ell}$ is an isomorphism from $M(r)_\ell$ to $M(r)_{\ell+2}$ and the diagram
	\[\begin{tikzcd}
	{L(r)} && {M(r)_\ell} && {M(r)_{\ell+2}} && {H(r)} \\
	\\
	{L(r)} && {M(r)_\ell} && {M(r)_{\ell+2}} && {H(r)}
	\arrow["{X\big|_{M(r)_\ell}}", from=1-3, to=1-5]
	\arrow["{k\big|_{M(r)_\ell}}"', from=1-3, to=3-3]
	\arrow["{k\big|_{M(r)_{\ell+2}}}", from=1-5, to=3-5]
	\arrow["{X\big|_{M(r)_\ell}}"', from=3-3, to=3-5]
	\arrow[from=1-5, to=1-7, dashed]
	\arrow[from=3-5, to=3-7, dashed]
	\arrow[from=1-1, to=1-3, dashed]
	\arrow[from=3-1, to=3-3, dashed]
	\arrow["{k\big|_{L(r)}}"', from=1-1, to=3-1]
	\arrow["{k\big|_{H(r)}}", from=1-7, to=3-7]
	\end{tikzcd}\]
	commutes. Then $k\big|_{M(r)_\ell}$ has the same determinant for all $\ell$. Since $X$ has only even parts, the number of weight spaces in $M(r)$ is even for each $r$. Then 
	$$\det k=\prod_r\det\left( k\big|_{M(r)}\right)=\prod_r\prod_\ell  \det\left( k\big|_{M(r)_\ell}\right)=\prod_r\left(\det\left( k\big|_{H(r)}\right)\right)^\ell=1$$
	as desired.
\end{proof}
	
	Now suppose that $X,X'\in \mathcal{A}$ have the same partition and contain only even parts. If $\det g=1$, they are in the same $\mathrm{SO}_{2n}$-orbits. Suppose that $\det g=-1$ and they are conjugated by another element $h\in \mathrm{SO}_{2n}$. Say $g\cdot X=X' = h\cdot X$ and let $k=g^{-1}h$. Then $\det k=(\det g^{-1})(\det h)=-1$ but this contradicts to Lemma \ref{lemma:stabilizer}. In this case, it means $X,X'$ are conjugated by an element in $\mathrm{O}_{2n}$ of determinant $-1$ only. We have the following theorem:
	
\begin{theorem}
\label{theorem:SO_{2n}}
	Nilpotent $\mathrm{SO}_{2n}$-orbits in $\mathcal{A}$ are parametrized by partitions of $2n$ except that the partitions with only even parts correspond to two orbits.
\end{theorem}

	For each nilpotent element $X$ having the partition $[d_1,d_2,...,d_k]$, we denote the $I_B$-orbits of $X$ by $\mathcal{O}_X$,$\mathcal{O}_{[d_1,d_2,...,d_k]}$, or simply $[d_1,d_2,...,d_k]$. We are now ready to compute the dimension of nilpotent $I_B$-orbits. 

\begin{theorem}\label{theorem:dimension}
	Let $X$ be a nilpotent element in $\mathcal{A}$. Then the dimension of $I_B$-orbit of $X$ is
	$$\dim \mathcal{O}_X=\frac{1}{2}\left(m^2-\sum_i s_i^2\right).$$

\end{theorem}
\begin{proof}
	Suppose that $B$ is symplectic on $\mathbb{C}^m$, $m=2n$. Recall that we have the decomposition (\ref{eq:decomp}). For each $Z\in \mathfrak{g}_B^X$, we investigate how $Z$ sends $M(d)$. We consider $Z(M(d))$ and project it onto $M(e)$ for each $e\geq 0$. Since $Z$ and $X$ commute, by theory of representations of $\mathfrak{sl}_2$, the projection of $Z(M(d))$ onto $M(e)$ is uniquely determined by a linear map $L(d)\to M(e)$ where $L(d)$ is the lowest weight space in $M(d)$. 
	
	Suppose that a linear map from $L(d)$ to $M(e)$ is determined for $d<e$. Since $Z$ is skew-adjoint, $Z^*=-Z$ where $Z^*$ is a conjugate transpose of $Z$. Hence the projection of $Z(M(e))$ onto $M(d)$ is uniquely determined. Therefore, we only consider the case $d\leq e$.
	
	Now, consider the case $d<e$. For $v\in L(d)$, $X^{d+1}v=0$ and then $X^{d+1}Zv=ZX^{d+1}v=0$. Hence the nonzero $M(e)$-component of $Zv$ must lie in the weight spaces 
	$$M(e)_{e-2d}\oplus M(e)_{e-2(d-1)}\oplus \cdots \oplus M(e)_{e-2}\oplus M(e)_{e}$$ 
	where $M(e)_k$ is the $k$-weight space in $M(e)$. 
	Note that $r_{d+1}=\dim L(d)$. Therefore the set of all linear maps from $L(d)$ to $M(e)$ forms a vector space of dimension $(d+1)r_{d+1}r_{e+1}$.
	
	Assume that $d=e$. If $Z$ sends $L(d)$ to $H(d)$, then
	we define a new bilinear form $(\cdot,\cdot)_d$ on $L(d)$ given by $(u,v)_d=\langle u,Zv\rangle$. It can be checked $(\cdot, \cdot)_d$ is symmetric and completely determine the action of $Z$ on $L(d)$. The set of all such $(\cdot, \cdot)_d$ forms a vector space of dimension $\frac{1}{2}r_{d+1}(r_{d+1}+1)$. If $Z$ sends $L(d)$ to $(d-2)$-weight space in $M(d)$, we define the new form by $(u,v)_{d-2}=\langle u,XZv\rangle$. Again, this form is symmetric and completely determine the action of $Z$ on $L(d)$. Continue this process up to the case $Z$ sends $L(d)$ to itself.
	We obtain
	\begin{align*}
	\dim\mathfrak{g}_B^X&=\sum_{d\geq 0}\left[(d+1)\left(\sum_{e>d}r_{d+1}r_{e+1}\right)+\dfrac{d+1}{2}r_{d+1}(r_{d+1}+1)\right]\\
	&=\left[r_1(r_2+r_3+\cdots)+\dfrac{1}{2}r_1(r_1 +1)\right]+\left[2r_2(r_3+r_4+\cdots)+\dfrac{2}{2}r_2(r_2 +1)\right]\\
	&\hspace{2 cm} +\left[3r_3(r_4+r_5+\cdots)+\dfrac{3}{2}r_3(r_3 +1)\right]+\cdots\\
	&=\left[\dfrac{1}{2}r_1(r_1+2r_2+2r_3+...)+\dfrac{1}{2}r_1\right]+\left[\dfrac{2}{2}r_2(r_2+2r_3+2r_4+...)+\dfrac{2}{2}r_2\right]\\
	&\hspace{2 cm} +\left[\dfrac{3}{2}r_3(r_3+2r_4+2r_5+...)+\dfrac{3}{2}r_3\right]+\cdots\\
	&=\left[\dfrac{1}{2}(s_1-s_2)(s_1+s_2)+\dfrac{1}{2}r_1\right]+\left[\dfrac{2}{2}(s_2-s_3)(s_2+s_3)+\dfrac{2}{2}r_2\right]\\
	&\hspace{2 cm} +\left[\dfrac{3}{2}(s_3-s_4)(s_3+s_4)+\dfrac{3}{2}r_3\right]+\cdots\\
	&=\dfrac{1}{2}\sum_i s_i^2+\dfrac{1}{2}(r_1+2r_2+3r_3+\cdots)+\cdots\\
	&=\dfrac{1}{2}\sum_i s_i^2+\dfrac{1}{2}\sum_i s_i.\\
	&=n+\dfrac{1}{2}\sum_i s_i^2
	\end{align*}
	and hence
	$$\dim\mathcal{O}_X=\dim \mathfrak{g}_B-\dim\mathfrak{g}_B^X=(2n^2+n)-\left(n+\dfrac{1}{2}\sum_i s_i^2\right)=\dfrac{m^2}{2}-\frac{1}{2}\sum_i s_i^2.$$
	
	If $B$ is symmetric, the arguement is similar except that the form $(u,v)_d$ is symplectic and hence the vector space consisting of such forms $(\cdot,\cdot)_d$ has dimension $\frac{1}{2}r_{d+1}(r_{d+1}-1)$.
\end{proof}
\begin{remark}
The dimension of $I_B$-orbits can also be obtained from \cite[3.1.c, 3.2.b]{Se}, where the formulae are not uniform and the proofs are also omitted. 
\end{remark}

	The closure relation on the set of nilpotent orbits in $\mathcal{A}$ is given by 
	$$\mathcal{O}_Y \preceq \mathcal{O}_X \text{ if }\mathcal{O}_Y \subset \overline{\mathcal{O}}_X$$
	for nilpotent elements $X,Y \in \mathcal{A}$.
	Given two partitions $\bar{d}=[d_1,...,d_N], \bar{f}=[f_1,...,f_N]$ of $N$ (put some $d_i,f_i=0$ if needed). We say that $\bar{d}$ dominates $\bar{f}$, denoted by $\bar{d} \succeq \bar{f}$ if
	\begin{center}
		$d_1\geq f_1$\\
		$d_1+d_2\geq f_1+f_2$\\
		$\vdots$\\
		$d_1+...+d_N \geq f_1+...+f_N.$
	\end{center}
	
\begin{theorem}
	Let $X,Y$ be nilpotent elements in $\mathcal{A}$ having partition $\bar{d},\bar{f}$, respectively.
	Then $\mathcal{O}_{\bar{d}} \succeq \mathcal{O}_{\bar{f}}$ if and only if $\bar{d}$ dominates $\bar{f}$. 	
\end{theorem}
\begin{proof}
	See \cite[Theorem 1]{Oh}.
\end{proof}
	For example, all nilpotent $\mathrm{Sp}_{10}$-orbits in $\mathcal{A}\subset \mathfrak{sl}_{10}$ are
	$$\mathcal{O}_{[5^2]}\succeq \mathcal{O}_{[4^2,1^2]}\succeq \mathcal{O}_{[3^2,2^2]}\succeq \mathcal{O}_{[3^2,1^4]}\succeq \mathcal{O}_{[2^4,1^2]}\succeq \mathcal{O}_{[2^2,1^6]}\succeq \mathcal{O}_{[1^{10}]}.$$
	The dimensions are 40, 36, 32, 28, 24, 16, 0, respectively.


\section{The connection between Schubert cells and nilpotent $K$-orbits}
\label{section:Schubert_Nilpotent}
	The goal of this section is to show that for any small $\bar{\lambda}$, $\mathcal{M}_{\bar{\lambda}}$ is sent to the nilpotent cone $\mathcal{N}_\mathfrak{p}$ by the map $\pi$, and show that how each $\mathcal{M}_{\bar{\lambda}}$ is sent to nilpotent $K$-orbits in the $\mathcal{N}_\mathfrak{p}$. Theorem \ref{theorem:cellandorbit} describes the image $\pi(\mathcal{M}_{\bar{\lambda}})$ where the proofs are provided by case-by-case consideration in this section.
	
	Let $X_N$ be the type of Dynkin diagram of $G$ and $\sigma$ the diagram automorphism on $G$ of order $r$, denoted by the pair $(X_N,r)$. We consider the cases $(X_N,r)=(A_{2\ell},2),(A_{2\ell-1},2)$, and $(D_{\ell+1},2)$. Then $H=(\check{G})^\sigma$ is either of type $B_\ell$ or $C_\ell$. We make the following labelling for simple roots $\gamma_i$:  	\begin{equation}\label{equation:labelgamma}
	\begin{cases}
	\gamma_1=\gamma_{\{1,2\ell-1\}},\dots,\gamma_{\ell-1}=\gamma_{\{\ell-1,n+1\}},\gamma_{\ell}=\gamma_{\{\ell \}} &\quad   (X_N, r)=(A_{2\ell-1}, 2);\\
	\gamma_1=\gamma_{\{1,2\ell \}},\dots,\gamma_{\ell-1}=\gamma_{\{\ell-1,\ell+2\}},\gamma_\ell=\gamma_{\{\ell,\ell+1\}} &\quad   (X_N, r)=(A_{2\ell}, 2);\\
	\gamma_1=\gamma_{\{1\}},\dots,\gamma_{\ell-1}=\gamma_{\{\ell-1\}},\gamma_\ell=\gamma_{\{\ell,\ell+1\}} &\quad   (X_N, r)=(D_{\ell+1}, 2).\\
	\end{cases}
	\end{equation}	
	This labelling of vertices of type $B_\ell$ and $C_\ell$ agrees with the labelling in \cite[TABLE Fin, p.53]{Ka}. 
Then the highest short root $\gamma_{0}$ of $H$ can be described in the following table. 

\begin{table}[htb]
	\begin{tabular}{@{}ccccc@{}}
		\cmidrule(r){1-5}
		$(X_N,r)$  & $G$ & $H$  & Simple roots of $H$  & Highest short root $\gamma_0$ of $H$ \\ \cmidrule(r){1-5}
		$(A_{2\ell},2)$   & $\mathrm{SL}_{2\ell+1}$ & $\mathrm{PSO}_{2\ell+1}$ & \thead{$\gamma_i=\bar{\check{\alpha}}_i=\bar{\check{\alpha}}_{2\ell-i+1}$,\\ $i=1,...,\ell$} & $\gamma_1+\gamma_2+...+\gamma_\ell.$                   \\
		$(A_{2\ell-1},2)$ & $\mathrm{SL}_{2\ell}$   & $\mathrm{PSp}_{2\ell}$   &  \thead{$\gamma_i=\bar{\check{\alpha}}_i=\bar{\check{\alpha}}_{2\ell-i}$,\\ $i=1,...,\ell$}                                                                                & $\gamma_1+2\gamma_2+...+2\gamma_{\ell-1}+\gamma_\ell.$ \\
		$(D_{\ell+1},2)$  & $\mathrm{Spin}_{2\ell+2}$ & $\mathrm{PSO}_{2\ell+1}$ & \thead{$\gamma_i= \bar{\check{\alpha}}_i, i=1,...,\ell-1$\\ $\gamma_\ell=\bar{\check{\alpha}}_\ell =\bar{\check{\alpha}}_{\ell+1}$}    & $\gamma_1+\gamma_2+...+\gamma_\ell.$                   \\ \cmidrule(r){1-5}
	\end{tabular}
\end{table}

From Section \ref{subsection:Root datum}, we can identify the weight lattice of $H$ with $X_*(T)_\sigma$, and the set of dominant weights of $H$ with $X_*(T)^+_\sigma$. Then, from the construction of root system of classical Lie algebras given in \cite[\S 12]{Hu2}, we can make the following identifications: 
	\[X_{*} (T)_{\sigma}\cong
	 \begin{cases}\label{equation:coweight}
	\mathbb{Z}^\ell & \text{if } H=B_\ell;\\
	\{(a_1,...,a_\ell)\in\mathbb{Z}^\ell\mid a_1+ \cdots + a_\ell\in 2\mathbb{Z}\} & \text{if } H=C_\ell
	\end{cases}
	\]
	and 
	$$X_{*} (T)_{\sigma}^+\cong\{(a_1,...,a_\ell)\in X_{*} (T)_{\sigma}\mid a_1\geq \cdots \geq a_\ell\geq 0\}$$
	for any cases of $H$. This identification preserves the relation on $X_{*} (T)_{\sigma}^+$ and the dominance relation on $\{(a_1,...,a_\ell)\in X_{*} (T)_{\sigma}\mid a_1\geq \cdots \geq a_\ell\geq 0\}$.
	
In the Table \ref{table:tuples}, we can further make the following identifications for simple roots and fundamental weights of $H$. Those fundamental weights follows from \cite[Table 1., p.69]{Hu2}.
\begin{table}
	\begin{tabular}{|c|c|c|c|}
		\hline
		$H$      & Simple roots                                                                                                                                & $\gamma_0$ & Fundamental weights \\ \hline
		$B_\ell$ & \thead{$\gamma_1=(1,-1,0,0,...,0)$\\ $\gamma_2=(0,1,-1,0,...,0)$\\ $\vdots$ \\ $\gamma_{\ell-1}=(0,0,0,...,1,-1)$\\ $\gamma_\ell=(0,0,0,...,0,1)$} &  $(1,0,0,...,0,0)$  &  \thead{$\omega_j=(1^j 0^{\ell-j})$, $j=0,...,\ell-1$\\ $\omega_\ell=(\dfrac{1}{2},\dfrac{1}{2},...,\dfrac{1}{2})$}  \\ \hline
		$C_\ell$ & \thead{$\gamma_1=(1,-1,0,0,...,0)$\\
			$\gamma_2=(0,1,-1,0,...,0)$\\
			$\vdots$\\
			$\gamma_{\ell-1}=(0,0,0,...,1,-1)$\\
			$\gamma_\ell=(0,0,0,...,0,2).$} &  $(1,1,0,...,0,0)$  &   $\omega_j=(1^j 0^{\ell-j}),j=0,...,\ell$  \\ \hline
	\end{tabular}
	\caption{\label{table:tuples}Simple roots, highest short root, and fundamental weights of $H$ in term of tuples}
\end{table}

The following lemma is well-known, cf.\cite{AH}. We give a self-contained proof here. 
\begin{lemma}\label{lemma:smalldom}
	All small dominant weights of $H$ are
	\begin{enumerate}
		\item $\omega_j=(1^j 0^{\ell-j})$, $j=0,...,\ell-1$, $2\omega_\ell=(1,1,...,1)$, if $H$ has the type $B_\ell$.
		\item $\omega_1+\omega_{2j+1}=(21^{2j} 0^{\ell-2j-1})$, $j=0,...,\lfloor \frac{\ell-1}{2}\rfloor$\\
		$\omega_{2j}=(1^{2j} 0^{\ell-2j})$, $j=0,...,\lfloor \frac{\ell}{2}\rfloor$, if $H$ has the type $C_\ell$.
	\end{enumerate}
\end{lemma}
\begin{proof}	Suppose that $H$ has the type $B_\ell$. The highest short root is $\gamma_0=(1,0,...,0)$. By definition, a dominant weight $(a_1,...,a_\ell)\in X_{*} (T)_{\sigma}^+$ is small if and only if  $(a_1,...,a_\ell)\nsucceq (2,0,...,0)$ which is equivalent to $a_1\leq 1$. This proves the first part.
	
	Now assume that $H$ has the type $C_\ell$. The highest short root is $\gamma_0=(1,1,...,0)$. Let $(a_1,...,a_\ell)\in X_{*} (T)_{\sigma}^+$ be a small dominant weight. Then $(a_1,...,a_\ell)\nsucceq (2,2,...,0)$ and so $a_1\leq 2$. If $a_1=1$, then $(a_1,...,a_\ell)=(1^{2j} 0^{\ell-2j})$. If $a_1=2$, then $a_2<2$ and hence $(a_1,...,a_\ell)=(21^{2j} 0^{\ell-2j-1})$.
\end{proof}

 Let $\bar{\mu}$ be the maximal element among all small dominant weights of $H$, then
	$$\mathcal{G}r_{\mathrm{sm}}=\coprod_{{\bar{\lambda}}\preceq \bar{\mu}, \\\bar{\lambda} \text{  small}}\mathcal{G}r_{\bar{\lambda}}=\overline{\mathcal{G}r}_{\bar{\mu}}.$$
	Since $\mathcal{G}r_{\mathrm{sm}}$ is irreducible and $\mathcal{M}$ is an open subset of $\mathcal{G}r_{sm}$, $\mathcal{M}$ is irreducible. 

The following theorem is the main result of this section. 
\begin{theorem}\label{theorem:cellandorbit}
If $\bar{\lambda}$ is small, then $\pi(\mathcal{M}_{\bar{\lambda}} )$ is contained in $\mathcal{N}_{\mathfrak{p}}$. Moreover, the image $\pi(\mathcal{M}_{\bar{\lambda}} )$ can be described as the union of nilpotent orbits as the following table:  \vspace{0.2em}
\begin{equation*}
	\begin{tabular}{|c|c|c|}
		\hline
		$(X_N,r)$ & Small dominant weight $\overline{\lambda}$ of $H$ & Orbits in $\pi(\mathcal{M}_{\bar{\lambda}})$ \\ \hline
		$(A_{2\ell},2)$ &  $(1^j 0^{\ell -j}), j=0,1,\dots,\ell$ & $[2^j 1^{2\ell-2j+1}]$ \\ \hline
		\multirow{5}{*}{$(A_{2\ell-1},2)$} & $(1^{2j}0^{\ell-2j}), j=0,1,\dots,\lfloor \frac{\ell}{2}\rfloor$ & $[2^{2j}1^{2\ell-4j}]$ \\ \cline{2-3} 
		& $(20^{\ell-1})$ & $0,[2^2 1^{2\ell-4}]$ \\ \cline{2-3} 
		& $(21^20^{\ell-3})$ & $[2^2 1^{2\ell-4}],[2^4 1^{2\ell-8}],[3^2 1^{2\ell-6}]$ \\ \cline{2-3}
		& $(21^{2j}0^{\ell-2j-1}),j=2,\dots,\lfloor \frac{\ell-3}{2}\rfloor$ & \thead{$[2^{2j}1^{2\ell-4j}],[2^{2j+2}1^{2\ell-4j-4}],$\\$[3^2 2^{2j-2} 1^{2\ell-4j-2}],[3^2 2^{2j-4} 1^{2\ell-4j+2}]$} \\ \cline{2-3} 
		& \multirow{2}{*}{$(21^{2\lfloor \frac{\ell-1}{2}\rfloor}0^{\ell-2\lfloor \frac{\ell-1}{2}\rfloor-1})$} & \thead{$[2^{\ell-2}1^{4}],[2^{\ell}],$\\$[3^2 2^{\ell-4} 1^{2}],[3^2 2^{\ell-6} 1^4]$}, if $\ell$ is even \\ \cline{3-3} 
		&  & \thead{$[2^{\ell-1}1^{2}]$,\\$[3^2 2^{\ell-3}],[3^2 2^{\ell-5} 1^4]$}, if $\ell$ is odd \\ \hline
		\multirow{2}{*}{$(D_{\ell+1},2)$}  & \multirow{2}{*}{$(1^j 0^{\ell-j}),j=0,1,\dots,\ell$}  & 
		$0$, if $j=0$ \\ \cline{3-3} &    & $0,[31^{2\ell-1}]$, if $j$ is even, $j\geq 2$ \\ \cline{3-3} 
		&                       & $[31^{2\ell-1}]$, if $j$ is odd \\ \hline   
	\end{tabular}
\end{equation*}  
where the nilpotent orbit $[a_{1}^{i_1},...,a_{r}^{i_r}]$ in the above table means empty set if the associated partition is invalid for some small $\ell$. 
\end{theorem}
This theorem follows from Theorem \ref{theorem: isoA2l}, Theorem \ref{theorem: isoA2l-1}, Theorem \ref{theorem:pi A2l-1}, and Theorem \ref{theorem:piD}, which will be proved separately case by case.

The partial orders of small dominant weights of $H$ are shown in the below picture, where the partial order is compatible with the height.

\begin{figure}[htb]
	\[
	\adjustbox{scale=0.75,center}{
	\begin{tikzcd}
	{2\omega_\ell} &&& {\omega_1+\omega_{2\lfloor\frac{\ell-1}{2}\rfloor+1}} \\
	{\omega_{\ell-1}} && {\omega_1+\omega_{2\lfloor\frac{\ell-1}{2}\rfloor-1}} \\
	{\omega_{\ell-2}} && {\omega_1+\omega_{2\lfloor\frac{\ell-1}{2}\rfloor-3}} && {\omega_{2\lfloor\frac{\ell}{2}\rfloor}} \\
	&&&& {\omega_{2\lfloor\frac{\ell}{2}\rfloor-2}} \\
	&& {\omega_1+\omega_3} \\
	{\omega_2} && {2\omega_1} && {\omega_4} \\
	{\omega_1} &&&& {\omega_2} \\
	{0} &&& {0}\\
	\text{Type } B_\ell &&& \text{Type } C_\ell
	\arrow[from=1-4, to=2-3, no head]
	\arrow[from=2-3, to=3-3, no head]
	\arrow[from=3-3, to=5-3, dashed, no head]
	\arrow[from=5-3, to=6-3, no head]
	\arrow[from=6-3, to=8-4, no head]
	\arrow[from=1-4, to=3-5, no head]
	\arrow[from=3-5, to=4-5, no head]
	\arrow[from=4-5, to=6-5, dashed, no head]
	\arrow[from=6-5, to=7-5, no head]
	\arrow[from=7-5, to=8-4, no head]
	\arrow[from=2-3, to=3-5, no head]
	\arrow[from=3-3, to=4-5, no head]
	\arrow[from=5-3, to=6-5, no head]
	\arrow[from=6-3, to=7-5, no head]
	\arrow[from=1-1, to=2-1, no head]
	\arrow[from=2-1, to=3-1, no head]
	\arrow[from=3-1, to=6-1, dashed, no head]
	\arrow[from=6-1, to=7-1, no head]
	\arrow[from=7-1, to=8-1, no head]
	\end{tikzcd}    }
	\]
\end{figure}

	We first recall a crucial lemma from \cite[Lemma 4.3]{AH}.
\begin{lemma}\label{lemma:miracleAH}
	Let $g=\sum_{i=N}^{\infty}x_i t^i \in \mathrm{SL}_n(\mathcal{K})$, where $x_N\neq 0$. Let ${\lambda}=(a_1,a_2,...,a_n)$ be a tuple of integers such that $a_1\geq a_2\cdots \geq a_n$ and $\sum a_i=0$, and  $g(t)\in \mathrm{SL}_n(\mathcal{O})t^{{\lambda}} \mathrm{SL}_n(\mathcal{O})$. Then
	\begin{enumerate}
		\item $N=a_n$.
		\item The rank of $x_N$ equals to the number of $j$ such that $a_j=a_n$.
		\item For any $s\geq 1$,
		\[\rk\begin{pmatrix}
		x_N & x_{N+1} & \cdots & x_{N+s-2} & x_{N+s-1}\\
		0 & x_{N} & \cdots & x_{N+s-3} & x_{N+s-2}\\
		\vdots & \vdots & \ddots & \vdots & \vdots\\
		0 & 0 & \cdots & x_N & x_{N+1}\\
		0 & 0 & \cdots & 0 & x_N
		\end{pmatrix}=\sum_{j=1}^{n}\max\{s-(a_j-a_n),0\}.\]
	\end{enumerate}
\end{lemma}

	We have the following lemma for the twisted version.
	
\begin{lemma}\label{lemma:miracle}
	Let $g(t)=\sum_{i=N}^{\infty}x_i t^i \in G(\mathcal{K})^\sigma$ where $x_i\in\Mat_{m\times m}$, $x_N\neq 0$. Let $\bar{\lambda}=(a_1,...,a_\ell)\in X_{*} (T)_{\sigma}^+$ be such that $g(t)\in G(\mathcal{O})^\sigma n^{\lambda} G(\mathcal{O})^\sigma$. Then
	\begin{enumerate}
		\item \[N=\begin{cases}
			-a_1 & \text{if } (X_N,r)=(A_{2\ell},2),(A_{2\ell-1},2);\\
			-2a_1 & \text{if } (X_N,r)=(D_{\ell+1},2).
		\end{cases}\]
		\item The rank of $x_N$ is equal to the number of $j$ such that $a_j=a_1$.
	\end{enumerate}
\end{lemma}
\begin{proof}
	We can write
	$$\bar{\lambda}=(a_1,...,a_\ell)=\sum_{i=1}^{\ell}\left(\sum_{j=1}^{i}a_j\right)\gamma_i$$
	where $\gamma_i$ are simple roots of $H$ as labelled by (\ref{equation:labelgamma}). We choose a representative $\lambda\in X_{*} (T)$ of $\bar{\lambda}$ by
	\[\lambda=\begin{cases}\label{equation:domcoweight}
	\sum_{i=1}^{\ell}\left(\sum_{j=1}^{i}a_j\right)\check{\alpha}_i & \text{if } (X_N,r)=(A_{2\ell},2),(D_{\ell+1},2);\\
	\sum_{i=1}^{\ell-1}\left(\sum_{j=1}^{i}a_j\right)\check{\alpha}_i+\dfrac{1}{2}\left(\sum_{j=1}^{\ell}a_j\right)\check{\alpha}_\ell & \text{if } (X_N,r)=(A_{2\ell-1},2)\\
	\end{cases}\]
	so that
	\[\lambda+\sigma\lambda=\begin{cases}
	\sum_{i=1}^{\ell}\left(\sum_{j=1}^{i}a_j\right)\check{\alpha}_i+\sum_{i=\ell+1}^{2\ell}\left(\sum_{j=1}^{2\ell-i+1}a_j\right)\check{\alpha}_i & \text{if } (X_N,r)=(A_{2\ell},2);\\
	\sum_{i=1}^{\ell}\left(\sum_{j=1}^{i}a_j\right)\check{\alpha}_i+\sum_{i=\ell+1}^{2\ell-1}\left(\sum_{j=1}^{2\ell-i}a_j\right)\check{\alpha}_i  & \text{if } (X_N,r)=(A_{2\ell-1},2);\\
	\sum_{i=1}^{\ell-1}\left(\sum_{j=1}^{i}2a_j\right)\check{\alpha}_i+\left(\sum_{j=1}^{\ell}a_j\right)(\check{\alpha}_\ell+\check{\alpha}_{\ell+1}) & \text{if } (X_N,r)=(D_{\ell+1},2).\\
	\end{cases}\]
	The simple coroots of ${G}$ are identified with tuples of integers through the construction of root system given from \cite[\S 12]{Hu2}. 
	
	Let $\rho: G\to {\rm GL}(V)$ be the standard representation of $G$.  	We will determine the double ${\rm SL}(V_\mathcal{O})$-coset in ${\rm SL}(V_\mathcal{K}) $ that  $\rho(g(t))$ belongs to.  
	
	If ${G}$ is of the type $A_m$, then  $\check{\alpha}_i$, $i=1,...,m$, are identified with the following $(m+1)$-tuples
	\begin{center}
		$\check{\alpha}_1=(1,-1,0,0,...,0)$\\
		$\check{\alpha}_2=(0,1,-1,0,...,0)$
		$$\vdots$$
		$\check{\alpha}_{m}=(0,0,0,...,1,-1)$   
	\end{center}
	and hence, as the coweight of $\mathrm{SL}_m$,  $\lambda+\sigma\lambda$  corresponds to the following tuples
	\[\lambda+\sigma\lambda=\begin{cases}
	(a_1,a_2,...,a_\ell,0,-a_\ell,...,-a_2,-a_1) & \text{if } (X_N,r)=(A_{2\ell},2);\\
	(a_1,a_2,...,a_\ell,-a_\ell,...,-a_2,-a_1) & \text{if } (X_N,r)=(A_{2\ell-1},2).
	\end{cases}\]
	Assume that ${G}$ has the type $D_{\ell+1}$. Then $\check{\alpha}_i$, $i=1,...,\ell+1$, are identified with following $(\ell+1)$-tuples
	\begin{center}
		$\check{\alpha}_1=(1,-1,0,0,...,0)$\\
		$\check{\alpha}_2=(0,1,-1,0,...,0)$
		$$\vdots$$
		$\check{\alpha}_{\ell}=(0,0,0,...,1,-1)$\\
		$\check{\alpha}_{\ell+1}=(0,0,0,...,1,1).$
	\end{center}
	Then $\lambda+\sigma\lambda=(2a_1,2a_2,...,2a_\ell,0)$
	as the coweight of $G=\mathrm{Spin}_{2\ell+2}$. Choose an appropriate maximal torus  and a positive root system in $G$.  Composing with $\rho: G\to {\rm SL}_{2\ell+2}$, 
as	the coweight of $\mathrm{SL}_{2\ell+2}$, $\lambda+\sigma\lambda$ corresponds to the following tuple
	$$(2a_1,2a_2,...,2a_\ell,0,0,-2a_\ell,...,-2a_2,-2a_1).$$

	We write $g(t)=A(t)n^\lambda B(t)$, where $n^\lambda$ is a norm of $t^\lambda$ defined by (\ref{equation:norm}) and $A(t),B(t)\in G(\mathcal{O})^\sigma$.  Hence $\rho(g(t)) \in \mathrm{SL}_m(\mathcal{O}) \rho ( t^{\lambda+\sigma\lambda} ) \mathrm{SL}_m(\mathcal{O}) $.  By the above description of  $\rho ( t^{\lambda+\sigma\lambda} )$, this lemma follows from Lemma \ref{lemma:miracleAH}.
\end{proof}

\subsection{Case $(X_N,r)=(A_{2\ell},2)$}
	Let $\langle \cdot , \cdot \rangle$ be a nondegenerate symmetric bilinear form on $V=\mathbb{C}^{2\ell+1}$ whose matrix is
	\[J=\begin{pmatrix}
	&  &  &  &  &  & 1\\ 
	&  &  &  &  & -1 & \\ 
	&  &  &  & 1 &  & \\ 
	&  &  & -1 &  &  & \\ 
	&  & \reflectbox{$\ddots$} &  &  &  & \\ 
	& -1 &  &  &  &  & \\ 
	1 &  &  &  &  &  & 
	\end{pmatrix}.\]
	The diagram automorphism $\sigma$ on $\mathfrak{g}$ given by (\ref{equation:sigma}) becomes $\sigma(A)=-JA^TJ^{-1}$ and the diagram automorphism $\sigma$ on $G$ is given by
\begin{equation}\label{equation:Dynkin}
	\sigma(A)=JA^{-T}J^{-1}.
\end{equation}
	This $\sigma$ gives the decomposition
	$\mathfrak{g}=\mathfrak{k}\oplus \mathfrak{p}$ to 1 and $-1$ eigenspaces $\mathfrak{k}$ and $\mathfrak{p}$, respectively. Let $K:=(\mathrm{SL}_{2\ell+1})^\sigma=\mathrm{SO}_{2\ell+1}$. The classification of nilpotent $K$-orbits in $\mathfrak{p}$ and their dimensions follow from the Theorem \ref{theorem:SO_{2n+1}} and \ref{theorem:dimension}.
	
 Set	
\begin{equation}
\label{order2_nil}
  \mathcal{N}_{\mathfrak{p}, 2 }=\{ x\in \mathcal{N}_\mathfrak{p} \,|\,   x^2=0  \}    .\end{equation}

\begin{theorem}\label{theorem: isoA2l}
$\pi$ maps $\mathcal{M}$ isomorphically onto $\mathcal{N}_{\mathfrak{p},2}$. Moreover,  $\pi$  maps $\mathcal{M}_{(1^j 0^{\ell-j})}$ isomorphically to $[2^j 1^{2\ell-2j+1}]$.  
\end{theorem}
\begin{proof}
We first show that $\pi$ maps injectively into  $\mathcal{N}_{\mathfrak{p},2}$. 	Let $g(t)\cdot e_0\in \mathcal{M}$. Then $g(t)\cdot e_0\in \mathcal{M}_{(1^j0^{\ell-j})}$ for some $j$. By Lemma \ref{lemma:miracle}, $g(t)=I+xt^{-1}$ for some $x\in \Mat_{2\ell+1, 2\ell+1}$. By Lemma \ref{lemma:iota}, $\iota(g(t)\cdot e_0)\in \mathcal{M}_{(1^j0^{\ell-j})}$. Hence $\iota(g(t))=(I-xt^{-1})^{-1}=I+zt^{-1}$ for some $z\in \Mat_{2\ell+1, 2\ell+1}$, and so $x^2=0$.  We now show the map  $\pi: \mathcal{M}\to \mathcal{N}_{\mathfrak{p},2}$ is bijective.  Let $x\in \mathcal{N}_{\mathfrak{p}}$ be such that $x^2=0$. Then $I+xt^{-1}\in G(\mathcal{K})^\sigma$. By (\ref{equation:Cartan}), $(I+xt^{-1})\cdot e_0\in \mathcal{G}r_{\bar{\lambda}}$ for some $\bar{\lambda}=(a_1,...,a_\ell)\in X_{*} (T)_{\sigma}^+$. By lemma $\ref{lemma:miracle}$, $\bar{\lambda}=(1^j 0^{\ell-j})$ for some $j$.  

Consider the map  $\phi:   \mathcal{N}_{\mathfrak{p}, 2 }\to  \mathcal{G}r^-_0$ given by $x\mapsto  (I+ xt^{-1})\cdot e_0$.  Clearly $\phi$ is a closed embedding, as  $\mathcal{G}r^-_0\simeq G(\mathcal{O}^-)^\sigma_0$ and $I+xt^{-1}\in G(\mathcal{O}^-)^\sigma_0$ if and only if $x$ is nilpotent.  By the  argument in the previous paragraph,  $\phi(\mathcal{N}_{\mathfrak{p}, 2 })= \mathcal{M}$ and $\pi\circ  \phi$ is the identity map on $\mathcal{N}_{\mathfrak{p},2}$. Thus, $\phi:   \mathcal{N}_{\mathfrak{p}, 2 }\to \mathcal{M}$ is an isomorphism, and $\pi: \mathcal{M}\to  \mathcal{N}_{\mathfrak{p}, 2 }$ is its inverse. 
	
Finally, we show $\pi$  maps $\mathcal{M}_{(1^j 0^{\ell-j})}$ isomorphically to ${[2^j 1^{2\ell-2j+1}]}$ for each $j$.  Since $\pi$ is $K$-equivariant, it suffices to show that $\pi$ maps $\mathcal{M}_{(1^j 0^{\ell-j})}$ onto 	${[2^j 1^{2\ell-2j+1}]}$. Let $(I+xt^{-1})\cdot e_0\in \mathcal{M}_{(1^j 0^{\ell-i})}$. Then $x^2=0$ and $x$ has the Jordan blocks of size at most 2. By Lemma $\ref{lemma:miracle}$, $\rk x=j$ and then $x$ has the partition $[2^j 1^{2\ell-2j+1}]$. It is obvious that $\pi$ is injective. To prove surjectivity, let $x\in \mathcal{N}_{\mathfrak{p}, 2 }$ having the partition $[2^j 1^{2\ell-2j+1}]$. Then, $(I+xt^{-1})\cdot e_0\in \mathcal{M}$ and hence $(I+xt^{-1})\cdot e_0\in \mathcal{M}_{(1^k 0^{\ell-k})}$ for some $k$. In fact, $k=j$ since $\pi((I+xt^{-1})\cdot e_0)=x\in [2^k 1^{2\ell-2k+1}].$
\end{proof}

\subsection{Case $(X_N,r)=(A_{2\ell-1},2)$}

	Let $\langle \cdot , \cdot \rangle$ be a symplectic bilinear form on $V=\mathbb{C}^{2\ell}$ whose matrix is
	\[J=\begin{pmatrix}
	&  &  &  &  &  & 1\\ 
	&  &  &  &  & -1 & \\ 
	&  &  &  & 1 &  & \\ 
	&  &  & -1 &  &  & \\ 
	&  & \reflectbox{$\ddots$} &  &  &  & \\ 
	& 1 &  &  &  &  & \\ 
	-1 &  &  &  &  &  & 
	\end{pmatrix}.\]
	The diagram automorphism $\sigma$ on $\mathfrak{g}$ given by (\ref{equation:sigma}) becomes $\sigma(A)=-JA^TJ^{-1}$ and the action  on $G$ is given by
	\begin{equation}\label{equation:Dynkin}
	\sigma(A)=JA^{-T}J^{-1}.
	\end{equation}
	This $\sigma$ gives the decomposition
	$\mathfrak{g}=\mathfrak{k}\oplus \mathfrak{p}$ to 1 and $-1$ eigenspaces $\mathfrak{k}$ and $\mathfrak{p}$, respectively. Let $K:=(\mathrm{SL}_{2\ell})^\sigma=\mathrm{Sp}_{2\ell}$. The classification of nilpotent $K$-orbits in $\mathfrak{p}$ and their dimensions follow from the Theorem \ref{theorem:classify} and \ref{theorem:dimension}.

	Define the following constructible sets
		$$\mathcal{M}':=\bigcup_{j=0}^{\lfloor \frac{\ell}{2}\rfloor}\mathcal{M}_{(1^{2j} 0^{\ell-2j})},\hspace{0.5 cm}
		\mathcal{M}'':=\bigcup_{j=0}^{\lfloor \frac{\ell-1}{2}\rfloor}\mathcal{M}_{(21^{2j} 0^{\ell-2j-1})}.$$
	By Lemma \ref{lemma:miracle}, the element of $\mathcal{M}'$ is of the form $(I+xt^{-1})\cdot e_0$ and the element of $\mathcal{M}''$ is of the form $(I+xt^{-1}+yt^{-2})\cdot e_0$. Let $\mathcal{N}_{\mathfrak{p},2}$ be the order 2 nilpotent cone defined as in (\ref{order2_nil}).
	\begin{theorem}
	\label{theorem: isoA2l-1}
	$\pi$ maps $\mathcal{M'}$ isomorphically onto $\mathcal{N}_{\mathfrak{p},2}$. Moreover, $\pi$ maps $\mathcal{M}_{(1^{2j} 0^{\ell-2j})}$ isomorphically to $[2^{2j} 1^{2\ell-4j}]$.
	\end{theorem}
	\begin{proof}
		The proof is the same as the proof of Theorem \ref{theorem: isoA2l}, where in this case we use Lemma \ref{lemma:miracle} for $(A_{2\ell-1}, 2)$. 
	\end{proof}
	Before we describe  elements of $\mathcal{M}''$, we need the following lemma.
\begin{lemma}\label{lemma:notfix}
	Let $g(t)\cdot e_0=(I+xt^{-1}+yt^{-2})\cdot e_0\in \mathcal{M}''$. Then
	\begin{enumerate}
		\item $\iota (g(t))\neq g(t)$
		\item If $g(t)\cdot e_0\in \mathcal{M}_{(21^{2j} 0^{\ell-2j-1})}$, then $\rk
			\begin{pmatrix}
				y & x\\
				0 & y
			\end{pmatrix}=2j+2.$
	\end{enumerate}
	 
\end{lemma}
\begin{proof}
	Note that $\iota(g(t))=g(t)$ if and only if
	$$(I+xt^{-1}+yt^{-2})(I-xt^{-1}+yt^{-2})=I$$
	which is equivalent to $y=\dfrac{1}{2}x^2$ and $x^4=0$. Suppose that $\iota (g(t))=g(t)$. Observe that $\rk x^3 \leq \rk x^2 = \rk y=1$. If $\rk x^3=1$, then $\rk x^4 =\rk x^3=1$ which is impossible. Hence $x^3=0$. Since $\rk x^2=1$, $x\in \mathfrak{p}$ is nilpotent having the partition $[32^j 1^{2l-2j-3}]$ but this contradicts to Theorem \ref{theorem:classify}. This proves the first part.
	
	Assume that $g(t)\cdot e_0\in \mathcal{M}_{(21^{2j} 0^{\ell-2j-1})}$. We write 
	$$g(t)=At^{(2, 1^{2j}, 0^{2\ell-4j-2}, (-1)^{2j}, -2)} B$$ 
	where $A=\sum_{i=0}A_i t^i, B=\sum_{i=0}B_i t^i\in G(\mathcal{O})^\sigma$. In particular, $g(t)\in \mathrm{SL}_{2\ell}(\mathcal{O})t^{{\lambda}}\mathrm{SL}_{2\ell}(\mathcal{O})$ where ${\lambda}=(2, 1^{2j}, 0^{2\ell-4j-2}, (-1)^{2j}, -2)$. By Lemma \ref{lemma:miracleAH},
	\[\rk\begin{pmatrix}
	y & x\\
	0 & y
	\end{pmatrix}=\sum_{j=1}^{2\ell}\max\{-a_j,0\}=2j+2\]
	as desired.

\end{proof}

	Let $g(t)=I+xt^{-1}+yt^{-2}$. By Lemma \ref{lemma:iota},  we can write $\iota(g(t))=I+x't^{-1}+y't^{-2}$ for some matrices $x',y'$. Hence
	$$(I-xt^{-1}+yt^{-2})(I+x't^{-1}+y't^{-2})=I=(I+x't^{-1}+y't^{-2})((I-xt^{-1}+yt^{-2})$$
	which implies
\begin{equation}\label{equation:conditiony'}
	x'=x,\hspace{0.25 cm} x^2=y+y',\hspace{0.25 cm} xy=y'x,\hspace{0.25 cm} xy'=yx,\hspace{0.25 cm} yy'=y'y=0.
\end{equation}
	By Lemma \ref{lemma:miracle} and Lemma \ref{lemma:notfix}, $\rk y=\rk y'=1$ and $y'\neq y$. Since $\sigma (g(t))=g(t)$, $y'=Jy^T J^{-1}$  which means that $y$ and $y'$ are adjoint to each other.
	We set
	$$\mathcal{M}_{(21^{2j} 0^{\ell-2j-1})}^{\mathrm{I}}:=\{(I+xt^{-1}+yt^{-2})\cdot e_0\in \mathcal{M}_{(21^{2j} 0^{\ell-2j-1})}\mid L=L'\},$$
	$$\mathcal{M}_{(21^{2j} 0^{\ell-2j-1})}^{\mathrm{II}}:=\{(I+xt^{-1}+yt^{-2})\cdot e_0\in \mathcal{M}_{(21^{2j} 0^{\ell-2j-1})}\mid L\neq L'\},$$
	where $L=\im y$ and $L'=\im y'$.

The following lemma  will be used in the proofs of Lemma \ref{lem4.9} and Theorem \ref{theorem:piD}.
\begin{lemma}\label{lemma:adjointmap}
	Let $(\cdot ,\cdot)$ be a nondegenerate symmetric or skew-symmetric bilinear form on a vector space $V$ over a field $\mathbb{C}$ and let $T$ a linear map on $V$. Denote the adjoint of $T$ by $T^*$. Assume that $\im T=\im T^*$ and $\rk T=1$. Then $T$ is self-adjoint or skew-adjoint.
\end{lemma}
\begin{proof}
	It is easy to see that $\ker T=(\im T^*)^\perp$ and $\ker T^*=(\im T)^\perp$. Say that $\im T=\mathbb{C}v$ and $\im T^*=\mathbb{C}v'$ for some $v,v'\in V$. Then $Tw=v$ for some $w\in V$. Since $\im T=\im T^*$, we have $T^*w=\lambda v$ for some $\lambda\in \mathbb{C}$. Let $u\in V$. Then $Tu=kv$ for some $k\in\mathbb{C}$. Since $T(u-kw)=0$,
	$u-kw\in \ker T= (\im T^*)^\perp=(\im T)^\perp=\ker T^*$. Hence $T^*u=T^*(kw)=k\lambda v=\lambda Tu$. Since $u$ is arbitrary, $T^*=\lambda T$. Consider $T+T^*=(1+\lambda)T$. Then
	$$(1+\lambda)T=T+T^*=(T+T^*)^*=(1+\lambda)T^*.$$
	Therefore $1+\lambda=0$ or $T=T^*$ which means that $T$ is skew-adjoint or self-adjoint. 
\end{proof}

\begin{lemma}
\label{lem4.9}
	If $(I+xt^{-1}+yt^{-2})\cdot e_0\in \mathcal{M}_{(21^{2j} 0^{\ell-2j-1})}^{\mathrm{I}}$, then $y'=-y$.
\end{lemma}
\begin{proof}
	We know that $y\neq y'$ are adjoint to each other, they have the same images, and $\rk y=1$. By Lemma \ref{lemma:adjointmap}, $y$ is skew-adjoint, i.e., $y'=-y$.
\end{proof}
	
\begin{theorem}\label{theorem:pi A2l-1}\hfill
	\begin{enumerate}
		\item If $\ell$ is even, then
		$$\pi(\mathcal{M}_{(21^{2j} 0^{\ell-2j-1})}^{\mathrm{I}})=[2^{2j} 1^{2\ell-4j}]\cup [2^{2j+2} 1^{2\ell-4j-4}]$$
		for $j=0,1,...,\frac{\ell-2}{2}$. If $\ell$ is odd, then
		\[\pi(\mathcal{M}_{(21^{2j} 0^{\ell-2j-1})}^{\mathrm{I}})=
		\begin{cases}
		[2^{2j} 1^{2\ell-4j}]\cup [2^{2j+2} 1^{2\ell-4j-4}] & \hspace{0.05cm} \text{if } \ell\geq 3, \, 0\leq  j \leq  \frac{\ell-3}{2};\\
		[2^{\ell-1} 1^{2}] & \hspace{0.05cm} \text{if } j=\frac{\ell-1}{2}.
		\end{cases}
		\]
		\item  When $\ell \geq 3$, for $j=1,..., \lfloor \frac{\ell-1}{2}\rfloor$, we have
		\[\pi(\mathcal{M}_{(21^{2j} 0^{\ell-2j-1})}^{\mathrm{II}})=
		\begin{cases}
		[3^2 1^{2\ell-6}] & \hspace{0.05cm} \text{if } j=1;\\
		[3^2 2^{2j-2} 1^{2\ell-4j-2}]\cup [3^2 2^{2j-4} 1^{2\ell-4j+2}] & \hspace{0.05cm} \text{if } \ell\geq 4, \, 2\leq  j  \leq \lfloor \frac{\ell-1}{2}\rfloor.
		\end{cases}
		\]
		Moreover, for any $\ell \geq 1$,  $\mathcal{M}_{(20^{\ell-1})}^{\mathrm{II}}$ is empty.
	\end{enumerate}
\end{theorem}
\begin{proof}
	Let $g(t)\cdot e_0=(I+xt^{-1}+yt^{-2})\cdot e_0\in\mathcal{M}_{(21^{2j} 0^{\ell-2j-1})}^{\mathrm{I}}$. By Lemma \ref{lemma:notfix},
	$$\rk x \leq \rk \begin{pmatrix}
	y & x\\
	0 & y
	\end{pmatrix}=2j+2\leq 2\rk y +\rk x=2+\rk x.$$
	Since $x^2=y+y'=0$ (by Lemma \ref{lem4.9}), $x$ is nilpotent whose partition is $[2^{2k}1^{2\ell-4k}]$ so that $\rk x=2k$. Hence $k=j$ or $j+1$. If $\ell$ is odd and $j=\frac{\ell-1}{2}$, then $k=j$.
	
	Let $E_{ij}\in \Mat_{2\ell \times 2\ell}$ be the matrix which has 1 at the entry $i,j$ and 0 elsewhere. 
	For each $j=0,1,\dots,\lfloor \frac{\ell-1}{2}\rfloor$, let
	$$x_j=\diag(0, J_2,..,J_2, 0_{2\ell-4j-2}, -J_2,.., -J_2, 0)$$
	where there are $j$ blocks of $J_2=\begin{pmatrix}
	0 & 1\\
	0 & 0
	\end{pmatrix}$ and $j$ blocks of $-J_2$, and $0_{2\ell-4j-2}$ is the square zero matrix of size $2\ell-4j-2$.
	Then $x_j\in\mathfrak{p}$ is nilpotent and has the partition $ [2^{2j} 1^{2\ell-4j}]$. 
	It is easy to check that $g(t)=I+x_j t^{-1}+E_{1,2\ell} t^{-2}\in G(\mathcal{K})^\sigma$ and $\iota(g(t))=I+x_jt^{-1}-E_{1,2\ell}t^{-2}$. By (\ref{equation:Cartan}), $g(t)\cdot e_0\in \mathcal{G}r_{\bar{\lambda}}$ for some $\bar{\lambda}=(a_1,...,a_\ell)\in X_*(T)_\sigma^+$ with $a_1\geq a_2 ... \geq a_\ell \geq0$ and $\sum a_i$ is even. By Lemma \ref{lemma:miracle}, since $\rk E_{1,2\ell}=1$, $\bar{\lambda}=(21^{2k} 0^{\ell-2k-1})$ for some $k$. By Lemma \ref{lemma:notfix},
	\[2k+2=\rk\begin{pmatrix}
	E_{1,2\ell} & x_j\\
	0 & E_{1,2\ell}
	\end{pmatrix}=2j+2.\]
	Then $(I+x_jt^{-1}+E_{1,2\ell}t^{-2})\cdot e_0\in \mathcal{M}_{(21^{2j} 0^{\ell-2j-1})}^\mathrm{I}$. For $j=0,1,\dots,\lfloor \frac{\ell-2}{2}\rfloor$, let
	$$x'_j=x_j+E_{1,2j+2}-E_{2\ell-2j-1,2\ell}.$$
	Then $x'_j\in\mathfrak{p}$ is nilpotent and has the partition $ [2^{2j+2} 1^{2\ell-4j-4}]$. Similarly, one can show that $(I+x'_jt^{-1}+E_{1,2\ell}t^{-2})\cdot e_0\in \mathcal{M}_{(21^{2j} 0^{\ell-2j-1})}^\mathrm{I}$. Since $\pi$ is $K$-invariant, this proves the first part.
	
	Now, we prove the second part. Let $g(t)\cdot e_0=(I+xt^{-1}+yt^{-2})\cdot e_0\in\mathcal{M}_{(21^{2j} 0^{\ell-2j-1})}^{\mathrm{II}}$. Set $U=L+L'$. Since $y\neq y'$, $\dim U=2$ and $U=\im x^2$. Assume that $L=\mathbb{C}v, L'=\mathbb{C}v'$. By (\ref{equation:conditiony'}), we have $xy=y'x$ and $xy'=yx$. Hence $xv=bv'$ and $xv'=av$ for some $a,b\in \mathbb{C}, v,v'\in\mathbb{C}^{2l}$. Then
	\[x\big|_U=
	\begin{pmatrix}
	0 & a\\
	b & 0
	\end{pmatrix},\hspace{0.5 cm}
	x^2\big|_U=
	\begin{pmatrix}
	ab & 0\\
	0 & ab
	\end{pmatrix}
	\]
	Suppose that $ab\neq 0$. We will show that $\langle\cdot , \cdot\rangle\big|_{U\times U}$ is nondegenerate.  Let $u$ be a vector in $ \mathbb{C}^{2\ell}$ such that $\langle x^2 u , x^2 v\rangle=0$ for all $v\in \mathbb{C}^{2\ell}$. Since $x^2$ is self-adjoint, $\langle x^4 u,v\rangle=0$ for all $v$. Therefore, $x^4u=0$ and so $x^2 u\in (\ker x^2\big)\cap U= \ker( x^2|_U) $. By the assumption that $ab\not=0$, $\ker( x^2|_U) =0$. Thus $x^2u=0$. This concludes that $\langle\cdot , \cdot\rangle\big|_{U\times U}$ is nondegenerate.
	
	 Since $yy'=0=y'y$, we have $L'\subset \ker y$ and $L\subset\ker y'$. Recall that $y,y'$ are adjoint to each other. It follows that $\ker y=(L')^{\perp}$ and $\ker y'=L^{\perp}$. Thus, $L'\subset (L')^\perp$ and $L\subset L^\perp$.  This implies $\langle v,v\rangle=\langle v',v'\rangle=0$. By the non-degeneracy of $\langle\cdot , \cdot\rangle\big|_{U\times U}$,   we must have $\langle v,v'\rangle\neq 0$. Since $x$ is self-adjoint with respect to the symplectic form $\langle , \rangle$, we have 
	$$ab\langle v,v'\rangle=\langle v,x^2 v'\rangle=\langle xv,xv'\rangle=\langle bv',av\rangle=-ab\langle v,v'\rangle$$
	which implies $ab=0$, a contradiction.  This shows that we must have $x^2=0$ on $U=\im x^2$, which means $x^4=0$. Since $x\in\mathfrak{p}$ and $\rk x^2=2$, by Theorem \ref{theorem:classify}, $x$ is nilpotent having the partition $[3^2 2^{2k} 1^{2\ell-4k-6}]$. By Lemma \ref{lemma:notfix},
	$$\rk x \leq \rk \begin{pmatrix}
	y & x\\
	0 & y
	\end{pmatrix}=2j+2\leq 2\rk y +\rk x=2+\rk x.$$
	Since $\rk x=2k+4$, $k=j-1$ or $j-2$. Here we see that $j\neq0$ and hence $\mathcal{M}_{(20^{\ell-1})}^{\mathrm{II}}$ is empty. When $j=1$, we see that $k=0$. Let
	\[J_3=\begin{pmatrix}
	0 & 1 & 0\\
	0 & 0 & 1\\
	0 & 0 & 0
	\end{pmatrix},\hspace{0.2 cm} J_2=\begin{pmatrix}
	0 & 1\\
	0 & 0
	\end{pmatrix}.\]
	For each $j=1,...,\lfloor \frac{\ell-1}{2}\rfloor$, let
	$$x_{j-1}=\diag(J_3,J_2,...,J_2,0_{2\ell-4j-2},-J_2,...,-J_2,-J_3)$$
	where there are $j-1$ blocks of $J_2$, and $j-1$ blocks of $-J_2$. Then $x_{j-1}\in\mathfrak{p}$ is nilpotent having the partition $[3^2 2^{2j-2} 1^{2\ell-4j-2}]$. Note that $g(t):=1+x_{j-1}t^{-1}+E_{13}t^{-2}\in G(\mathcal{K})^\sigma$ and $\iota(g(t))=1+x_{j-1}t^{-1}+E_{2\ell-2,2\ell}t^{-2}$. By (\ref{equation:Cartan}), $g(t)\cdot e_0\in \mathcal{G}r_{\bar{\lambda}}$ for some $\bar{\lambda}=(a_1,...,a_\ell)\in X_*(T)_\sigma^+$ with $a_1\geq a_2 ... \geq a_\ell \geq0$ and $\sum a_i$ is even. By Lemma \ref{lemma:miracle}, since $\rk E_{13}=1$, $\bar{\lambda}=(21^{2k} 0^{\ell-2k-1})$ for some $k$. By Lemma \ref{lemma:notfix},
	\[2k+2=\rk\begin{pmatrix}
	E_{13} & x_{j-1}\\
	0 & E_{13}
	\end{pmatrix}=2j+2.\]
	Then $(I+x_jt^{-1}+E_{13}t^{-2})\cdot e_0\in \mathcal{M}_{(21^{2j} 0^{\ell-2j-1})}^\mathrm{II}$. For $j=2,...,\lfloor \frac{\ell-1}{2}\rfloor$, let
	$$x'_{j-2}=\diag(0,J_3,J_2,...,J_2,0_{2\ell-4j},-J_2,...,-J_2,-J_3,0)$$
	where there are $j-2$ blocks of $J_2$, and $j-2$ blocks of $-J_2$. Then $x'_{j-2}\in\mathfrak{p}$ is nilpotent having the partition $[3^2 2^{2j-4} 1^{2\ell-4j+2}]$. One can check that $h(t):=1+x'_{j-2}t^{-1}+(E_{24}+E_{1,2\ell})t^{-2}\in G(\mathcal{K})^\sigma$ and $\iota(h(t))=1+x'_{j-2}t^{-1}+(E_{2\ell-3,2\ell-1}-E_{1,2\ell})t^{-2}$. Similarly, $h(t)\cdot e_0\in \mathcal{G}r_{\bar{\lambda}}$ where $\bar{\lambda}=(21^{2k} 0^{\ell-2k-1})$ for some $k$.  By Lemma \ref{lemma:notfix},
	\[2k+2=\rk\begin{pmatrix}
	E_{24}+E_{1,2\ell} & x'_{j-2}\\
	0 & E_{24}+E_{1,2\ell}
	\end{pmatrix}=2j+2.\]
	Then $(I+x'_{j-2}t^{-1}+(E_{24}+E_{1,2\ell})t^{-2})\cdot e_0\in \mathcal{M}_{(21^{2j} 0^{\ell-2j-1})}^\mathrm{II}$.
\end{proof}

In the following proposition, we describe the reduced fibers of $\pi: \mathcal{M}\to \pi (\mathcal{M})$. For any $x\in \pi (\mathcal{M})$, let $\pi^{-1}(x)_{\rm red}$ denote the reduced fiber of $\pi$ over $x$. 
\begin{prop}
\label{prop:fiber1}
	For any $x\in  \pi (\mathcal{M})$,  we have
	\begin{equation}
	\label{fiber1}
\pi^{-1}(x)_{\rm red}\cong\{z\in   \mathfrak{sp}_{2\ell}   \mid xz+zx=0, z^2=0, \rk(z+\frac{1}{2}x^2)\leq 1\}.  \end{equation}
	 In particular,
	$\pi^{-1}(0)_{\rm red }$ is isomorphic to the closure of nilpotent orbit $\mathcal{O}_{[21^{2\ell-2}]}$ in $\mathfrak{sp}_{2\ell}$ and $\dim\pi^{-1}(0)_{\rm red}=2\ell+1$.
\end{prop}
\begin{proof}
	Fix a nilpotent element $x$ in $ \pi (\mathcal{M})$. 	Note that $(1+xt^{-1}+yt^{-2})\cdot e_0\in \mathcal{M}$ if and only if $\det (1+xt^{-1}+yt^{-2})=1,\rk y\leq 1$ and
	\begin{equation}\label{equation:conditionA2l-1}
	x^TJ -Jx=0,\quad -x^TJ x+y^TJ + Jy=0,\hspace{0.5 cm} x^T Jy-y^T Jx=0,\hspace{0.5 cm} y^T Jy=0.
	\end{equation}
	Set $z=y-\frac{1}{2}x^2$, (\ref{equation:conditionA2l-1}) is equivalent to
	\begin{equation}\label{equation:z A2l-1}
	\quad z\in\mathfrak{k}=\mathfrak{sp}_{2\ell} ,\quad xz+zx=0,\quad z^2=0.
	\end{equation}
	When $xz+zx=0$ and $z^2=0$,  $xt^{-1}+(z+\frac{1}{2}x^2)t^{-2}$ is nilpotent in $\mathrm{Mat}_{2\ell \times 2 \ell}(\mathcal{K})$. Thus, $\det (1+xt^{-1}+(z+\frac{1}{2}x^2)t^{-2})=1$. Therefore, the isomorphism $(\ref{fiber1})$ holds.  In particular, when $x=0$ we have 
	\[  \pi^{-1}(0)_{\rm red}=\{ z\in \mathfrak{sp}_{2\ell} \mid  z^2=0,\rk z\leq 1\},  \]
and the dimension, cf.\,\cite[Corollary 6.1.4]{CM}, is given by
	$$\dim \pi^{-1}(0)_{\rm red}=\dim \mathcal{O}_{[21^{2\ell-2}]}= (2\ell^2+\ell)-\frac{1}{2}((2\ell-1)^2+1^2)-\frac{1}{2}(2\ell-2)=2\ell+1.$$
\end{proof}

\begin{figure}[htb]
	\[\hspace{2.4 cm}
	\adjustbox{scale=0.8}{
	\begin{tikzcd}
	& {\mathcal{M}_{(21^4)}^\mathrm{I}} \\
	& {\mathcal{M}_{(21^4)}^\mathrm{II}} &&&& {[3^2 2^2]} \\
	&&&&&&&&&&&&& {} \\
	&& {\mathcal{M}_{(21^2 0^2)}^\mathrm{I}} &&& {[3^2 1^4]} \\
	&& {\mathcal{M}_{(21^2 0^2)}^\mathrm{II}} \\
	{\mathcal{M}_{(1^4 0)}} &&&&& {[2^4 1^2]} \\
	&& {\mathcal{M}_{(20^4)}^\mathrm{I}} \\
	&& {\mathcal{M}_{(20^4)}^\mathrm{II}=\emptyset} &&& {[2^2 1^6]} \\
	\\
	& {\mathcal{M}_{(1^2 0^3)}} &&&& {0} \\
	\\
	& {0}
	\arrow[from=2-2, to=4-6]
	\arrow[from=10-2, to=6-1, no head]
	\arrow[from=5-3, to=4-6]
	\arrow[from=10-6, to=8-6, no head]
	\arrow[from=4-6, to=2-6, no head]
	\arrow[from=2-2, to=2-6]
	\arrow[from=4-3, to=6-6]
	\arrow[from=2-2, to=4-3, no head]
	\arrow[from=5-3, to=7-3, no head]
	\arrow[from=12-2, to=10-6]
	\arrow[from=10-2, to=12-2, no head]
	\arrow[from=2-2, to=6-1, no head]
	\arrow[from=5-3, to=6-1, no head]
	\arrow[from=8-3, to=10-2, no head]
	\arrow[from=4-6, to=6-6, no head]
	\arrow[from=6-6, to=8-6, no head]
	\arrow[from=1-2, to=6-6]
	\arrow[from=4-3, to=8-6]
	\arrow[from=7-3, to=10-6]
	\arrow[from=7-3, to=8-6]
	\arrow["{\sim}", from=6-1, to=6-6]
	\arrow["{\sim}", from=10-2, to=8-6]
	\end{tikzcd}   }    
	\]
	\caption{$\mathcal{M}_{\bar{\lambda}}$ for small dominant weight $\bar{\lambda}$ and their image under the map $\pi$ in type $(X_N,r)=(A_9,2)$}\label{figure:(A_9,2)}
\end{figure}
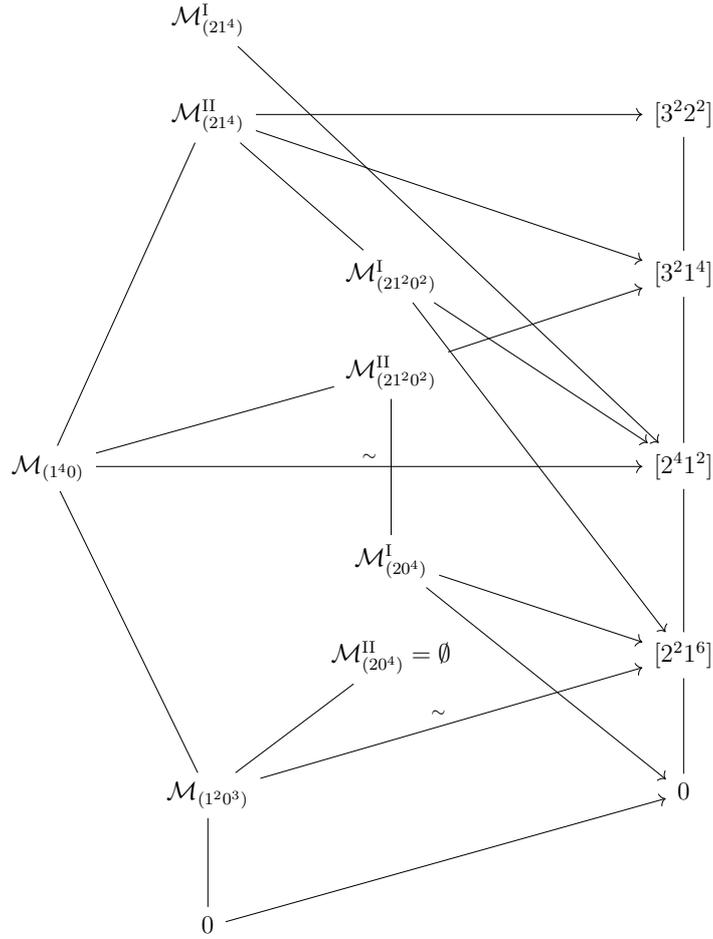

	In \cite[Theorem 1.2]{AH}, they proved that there are finitely many $G$-orbits in $\mathrm{Gr}_{\lambda}\cap \mathrm{Gr}_0^-$ for small dominant coweight $\lambda$. In the case of  $(A_{2\ell},2)$, it is easy to see that $K$ acts transitively on $\mathcal{M}_{(1^j0^{\ell-j})}$ and hence there are finitely many $K$-orbits in $\mathcal{M}$. For the case $(A_{2\ell-1},2)$, it is not obvious to determine if there are finitely many $K$-orbits in $\mathcal{M}_{(21^j 0^{\ell-j-1})}$. If $g(t)=1+xt^{-1}+(z+\frac{1}{2}x^2)t^{-2}\in\mathcal{M}_{(21^j 0^{\ell-j-1})}$, then $g(t)$ satisfies (\ref{equation:z A2l-1}).
	If the action of $K$ on the following anti-commuting nilpotent variety
	$$\{(x,z)\in \mathcal{N}_{\mathfrak{p}}\times \mathcal{N}_{\mathfrak{k}}\mid xz+zx=0\}$$
	by diagonal cojugation has finitely many orbits, then there are finitely many $K$-orbits in $\mathcal{M}_{(21^j 0^{\ell-j-1})}$.

\begin{example}
	Consider the case $(X_N,r)=(A_9,2)$ . In this case, $G=\mathrm{SL}_{10}$ and $\mathfrak{g}=\mathfrak{sl}_{10}=\mathfrak{k}\oplus \mathfrak{p}$. The diagram as shown in Figure \ref{figure:(A_9,2)} describes the image of $\mathcal{M}_{\bar{\lambda}}$ for each small dominant weight $\bar{\lambda}$. For instance, $\mathcal{M}_{(21^2 0^2)}$ consists of two parts, $\mathcal{M}_{(21^2 0^2)}^\mathrm{I}$ and $\mathcal{M}_{(21^2 0^2)}^\mathrm{II}$. By Theorem \ref{theorem:pi A2l-1}, $\pi(\mathcal{M}_{(21^2 0^2)}^\mathrm{I})$ is precisely the union of two nilpotent orbits $[2^4 1^2]$ and $[2^2 1^6]$ in $\mathfrak{p}$ while $\pi(\mathcal{M}_{(21^2 0^2)}^\mathrm{II})$ is the single nilpotent orbit $[3^2 1^4]$ in $\mathfrak{p}$. Since $(21^2 0^2)\succeq (1^4 0)$, $\mathcal{M}_{(21^2 0^2)}\succeq\mathcal{M}_{(1^4 0)}$. Similarly, $\mathcal{M}_{(21^2 0^2)}\succeq\mathcal{M}_{(2 0^4)}$. According to the table in Theorem \ref{theorem:cellandorbit}, the image of certain $\mathcal{M}_{\bar{\lambda}}$ is a union of 4 nilpotent orbits. It does not happen in this case since $\ell=5$ is not large enough. In the case of $(X_N,r)=(A_{13},2)$, $\pi(\mathcal{M}_{(21^4 0^2)})$ is a union of nilpotent orbits $[2^4 1^6], [2^6 1^2], [3^2 2^2 1^4], [3^2 1^8]$ in $\mathfrak{p}\subset \mathfrak{sl}_{14}$.
\end{example}

\subsection{Case $(X_N,r)=(D_{\ell+1},2)$}

In this case, it is more convenient to work with $G=\mathrm{SO}_{2\ell+2}$ and $\sigma$ is a diagram automorphism on $G$. It is known that $G^\sigma \simeq \mathrm{SO}_{2\ell+1}\times \{\pm I \}$.   
 Let $G(\mathcal{O})^{\sigma, \circ}   $ denote the identity component of the group $G(\mathcal{O})^{\sigma}$.   Then, the action of  ${\rm Spin}_{2\ell+2}(\mathcal{O})^\sigma$  on the twisted affine Grassmanian $\mathcal{G}r$ of   ${{\rm Spin}_{2n+2}}$ factors through $G(\mathcal{O})^{\sigma, \circ}   $. Let $G(\mathcal{O}^-)_0^\sigma$ be the kernel of the evaluation map $G(\mathcal{O}^- )^\sigma\to G^\sigma$. The action of ${\rm Spin}_{2\ell+2}(\mathcal{O}^-)_0$ on $\mathcal{G}r$ factors through $G(\mathcal{O}^-)_0^\sigma$.  Hence, the opposite open Schubert cell $\mathcal{G}r_0^-$ is a $G(\mathcal{O}^-)_0^\sigma$-orbit.    In fact, $\mathcal{G}r$ is naturally the neutral component of the twisted affine Grassmannian associated to $(G,\sigma)$, whose definition is a bit more involved.




We can realize the group $G$ as $\{g\in {\rm SL}_{2\ell+2} |  gJg^T=J \}$, and 
	the Lie algebra of $G$ as $\mathfrak{g}=\mathfrak{so}_{2\ell+2}(J)=\{x\in \mathfrak{gl}_{2n}\mid Jx+x^T J=0\}$ where
	\[J=\begin{pmatrix}
	&  &  &  &  & 1\\ 
	&  &  &  & 1 & \\ 
	&  &  & 1 &  & \\ 
	&  & \reflectbox{$\ddots$} &  &  &  & \\ 
	& 1 &  &  &  & \\ 
	1 &  &  &  &  & 
	\end{pmatrix}.\]
	The diagram automorphism $\sigma$ of order 2 on $\mathfrak{g}$ can be given by $\sigma(x)=wxw$ where
	$$w=\diag\left(I_{\ell},\begin{pmatrix}
	0 & 1\\
	1 & 0
	\end{pmatrix}, I_{\ell}\right).$$
	The diagram automorphism $\sigma$ on $G$ is also defined in the same way. We also have the decomposition $\mathfrak{g}=\mathfrak{k}\oplus \mathfrak{p}$. Let $K$ be the identity component of $G^\sigma$. $K$ has Lie algebra $\mathfrak{k}$ and acts on $\mathfrak{p}$ by conjugation. It can be checked that $J=A^T A$ where
	\[A=\begin{pmatrix}
	\frac{1}{\sqrt{2}} & 0 & \cdots & \cdots & 0 & \frac{1}{\sqrt{2}}\\ 
	0 & \ddots & \cdots & \cdots & \reflectbox{$\ddots$} & 0\\ 
	\vdots & \vdots & \frac{1}{\sqrt{2}} & \frac{1}{\sqrt{2}} & \vdots & \vdots\\ 
	\vdots & \vdots & \frac{i}{\sqrt{2}} & -\frac{i}{\sqrt{2}} & \vdots & \vdots\\ 
	0 & \reflectbox{$\ddots$} & \cdots  & \cdots & \ddots & 0\\ 
	\frac{i}{\sqrt{2}} & 0 & \cdots & \cdots & 0 & -\frac{i}{\sqrt{2}}
	\end{pmatrix}.
	\]
	Another realization of $\mathfrak{so}_{2\ell+2}$ is $\mathfrak{so}_{2\ell+2}(I)=\{x\in \mathfrak{gl}_{2\ell}\mid x+x^T =0\}$.
	There exists an isomorphism from $\mathfrak{so}_{2\ell+2}(J)$ to $\mathfrak{so}_{2\ell+2}(I)$ given by $x\mapsto AxA^{-1}$. Under $\mathfrak{so}_{2\ell+2}(I)$, the diagram automorphism $\sigma_0$ is defined by $\sigma_0(x)=w_0 x w_0$ where $w_0=(PA)w(PA)^{-1}=\diag(-1,1,1,...,1)$ and $P$ is some matrix of change of basis.
\begin{prop}\label{prop:nilpotentD}\hfill
	\begin{enumerate}
		\item If $x$ is a nonzero nilpotent element in $\mathfrak{p}$, then $x$ has the partition $[31^{2\ell-1}]$.
		\item There are exactly 2 nilpotent $K$-orbits in $\mathfrak{p}$: $\{0\}$ and $\mathcal{N}_\mathfrak{p}\setminus\{0\}.$
	\end{enumerate}
	
\end{prop}
\begin{proof}
	Since $w_0 x w_0 =-x$, $x$ has the form
	\[x=\begin{pmatrix}
	0 & -u^t \\
	u &  0
	\end{pmatrix}
	\]
	where $u\in \mathbb{C}^{2l+1}$ is a nonzero column vector. Then
	\[x^2=\begin{pmatrix}
	-u^t u & 0 \\
	0 & -uu^t
	\end{pmatrix}.\]
	If $x^2=0$, then $uu^t=0$ which implies $u=0$, a contradiction. Since $\rk x=2$ and $x^2\neq 0$, $x$ has the partition $[31^{2n-1}]$. This proves the first part.
	
	
	The element of $K$ has the form
	\[k=\begin{pmatrix}
	1 & 0 \\
	0 & g
	\end{pmatrix}.\]
	where $g\in \mathrm{SO}_{2\ell+1}$ and $k$ acts on $x\in \mathfrak{p}$ by
	\[k\cdot x=kxk^{-1}=\begin{pmatrix}
	1 & -(gu)^t \\
	gu & 0
	\end{pmatrix}.\]
	Hence the action of $K$ on $\mathfrak{p}$ is the same as the action of $\mathrm{SO}_{2\ell+1}$ on $\mathbb{C}^{2\ell+1}$. Note that for every $k\geq 3$, $x^k$ has the scalar $u^T u$ on every nonzero entry. Since $x$ is nilpotent, $u^T u=0$. The result immediately follows since the action of $\mathrm{SO}_{2\ell+1}$ on $\{z\in \mathbb{C}^{2\ell+1}\mid z^T z=0\}$ has two orbits.
\end{proof}

\begin{theorem}\label{theorem:piD}
	For $j=0,1,...,\ell$,  we have 
		     \[  \pi(\mathcal{M}_{(1^j 0^{\ell-j})})   =   \begin{cases}    
		     \mathcal{N}_{\mathfrak{p}}\setminus\{0\}     &\quad  \text{ if }  j \text{ is odd}; \\
 		       \{0\}   &\quad \text{ if } j=0;  \\
		              \mathcal{N}_{\mathfrak{p}} &\quad  \text{ if } j \text{ is even and } j \geq 2.  \end{cases}    .\]   
		Moreover,    $\mathcal{M}_{(1 0^{\ell-1})}=\{(1+xt^{-1}+\frac{1}{2}x^2t^{-2})\cdot e_0\mid x\in \mathcal{N}_{\mathfrak{p}}\setminus\{0\}\}$. Consequently, $\pi$ maps $\mathcal{M}_{(1 0^{\ell-1})}$ isomorphically onto $\mathcal{N}_{\mathfrak{p}}\setminus\{0\}$. 
\end{theorem}
\begin{proof}
	By Lemma \ref{lemma:miracle}, let $g(t)\cdot e_0=(1+xt^{-1}+yt^{-2})\cdot e_0\in \mathcal{M}_{(1^j 0^{\ell-j})}.$ We work under the realizations $\mathfrak{so}_{2\ell+2}(I)$ and $\mathrm{SO}_{2\ell+2}(I)$. Since $g(t)$ is fixed by $\sigma_0$, $w_0 xw_0=-x$ and $w_0 yw_0=y$. Similar to the proof of Proposition \ref{prop:nilpotentD}, $x,y$ are in the form
	\begin{equation}\label{equation:xy type D}
	x=\begin{pmatrix}
	0 & -u^T \\
	u &  0
	\end{pmatrix},\hspace{0.5 cm}
	y=\begin{pmatrix}
	y_0 & 0 \\
	0 &  D
	\end{pmatrix}
	\end{equation}
	where $u\in \mathbb{C}^{2\ell+1}$ is a column vector, $y_0\in\mathbb{C}$, and $D\in\Mat_{(\ell-1)\times(\ell-1)}$. Since $g^T g=I$, the following equations hold:
	\begin{equation}\label{equation:conditionsD}
	x^T +x=0,\hspace{0.5 cm} x^T x+y^T + y=0,\hspace{0.5 cm} x^T y+y^T x=0,\hspace{0.5 cm} y^T y=0.
	\end{equation}
	Then $y_0=0$ and $u^T u=0$ which implies $x^3=0$. Hence $x\in\mathcal{N}_{\mathfrak{p}}$.
	
	Suppose that $j$ is odd and $x=0$. We have $y+y^T=0$ and $y^T y=0$. Then $y_0=0$ and $D$ is a nilpotent element of $\mathfrak{so}_{2\ell+1}$ with $D^2=0$. Since $\rk D=\rk y=j$, $D$ has the partition $[2^j1^{2\ell-2j+2}]$. This contradicts the classification of nilpotent orbits of type B, \cite[Theorem 5.1.2]{CM}. Hence $x\neq 0$. 
	Now, consider the matrix $x_0\in \mathcal{N}_{\mathfrak{p}}\setminus\{0\}$ defined by
	\[x_0=\begin{pmatrix}
	0 & \cdots & 0 & -1 & -i \\
	\vdots &  & &  & \\
	0 & & &  & \\
	1 & & & & \\
	i & & & & 
	\end{pmatrix}.\]
	Let $N$ be a nilpotent element in $\mathfrak{so}_{2\ell-1}(I)$ having the partition $[2^{j-1}1^{2\ell-j+1}].$ Such a matrix $N$ exists in view of \cite[Theorem 5.1.2]{CM}. Then the matrices $x_0$ and 
	$$y_0:=\diag(0,...,0,N,0,0)+\dfrac{1}{2}x_0^2$$
	satisfy the relations in (\ref{equation:conditionsD}). Hence $g(t):=1+x_0t^{-1}+y_0t^{-2}\in G(\mathcal{O})^\sigma$. By (\ref{equation:Cartan}), $g(t)\cdot e_0\in \mathcal{G}r_{\bar{\lambda}}$ for some $\bar{\lambda}=(a_1,a_2,...,a_{\ell})\in X_{*} (T)_{\sigma}^+$ with $a_1\geq...\geq a_{\ell}$. Since $\rk y_0=j$, by Lemma \ref{lemma:miracle}, $\bar{\lambda}=(1^j 0^{\ell-j})$ and hence $g(t)\cdot e_0 \in \mathcal{M}_{(1^j 0^{\ell-j})}$. Since $\pi$ is $K$-equivariant, the first part is done. By $K$-equivariance, the second part also follows. 
	 
	
	Suppose that $j$ is even. For each $j=0,2,4,...,2\lfloor\frac{\ell}{2}\rfloor$, consider the $\ell\times \ell$ matrix
	\[\begin{pmatrix}
	&  & 1 &  &  &  &  & \\ 
	& \reflectbox{$\ddots$} &  &  &  &  &  & \\ 
	1 &  &  &  &  &  &  & \\ 
	&  &  &  &  &  &  & \\ 
	&  &  &  &  &  &  & \\ 
	&  &  &  &  &  &  & -1\\ 
	&  &  &  &  &  & \reflectbox{$\ddots$} & \\ 
	&  &  &  &  & -1 &  & 
	\end{pmatrix}\]
	where there are $\frac{j}{2}$ copies of each 1 and -1. Denote $z_j$ the square zero matrix of size $2\ell+2$ whose $\ell\times \ell$ submatrix on the right top is replaced by the above matrix. Now we work under $\mathfrak{so}_{2\ell+2}(J)$ and $\mathrm{SO}_{2\ell+2}(J)$. Since $wz_jw=z_j$ and $\rk z_j=j$, we have $(1+z_j t^{-2})\cdot e_0\in \mathcal{M}_{(1^j 0^{\ell-j})}$ and then $\pi((1+z_j t^{-2})\cdot e_0)=0$. Let $x_0$ be the square zero matrix of size $2\ell+2$ whose $4\times 4$ submatrix at the center is replaced by
	\[\begin{pmatrix}
	0 & 1 & -1 & 0\\ 
	0 & 0 & 0 & 1\\ 
	0 & 0 & 0 & -1\\ 
	0 & 0 & 0 & 0
	\end{pmatrix}.\]
	Then $x_0\in\mathcal{N}_{\mathfrak{p}}\setminus\{0\}$, 
	Set $y_0=\frac{1}{2}x_0^2 +z_j$. Then $\rk y_0=j$. It can be checked that $x_0,y_0$ satisfy $wx_0w=-x_0, wy_0w=y_0$, and
	\begin{equation}
	\begin{aligned}
	& x_0^T J +Jx_0=0, \hspace{0.5 cm}  x_0^T Jx_0+y_0^T J+ Jy_0=0,\\
	& x_0^TJ y_0+y_0^T Jx_0=0, \hspace{0.5 cm}  y_0^T Jy_0=0.
	\end{aligned}
	\end{equation}
	Hence $h(t):=1+x_0t^{-1}+y_0t^{-2}\in G(\mathcal{O})^\sigma$. Similarly, one can show that  $h(t)\cdot e_0\in \mathcal{M}_{(1^j 0^{\ell-j})}$. This proves the second part.
	
	To prove the last part, let $x$ be a nonzero nilpotent element in $\mathfrak{p}$. Since $x$ has the partition $[31^{2\ell}]$, $\rk x^2=1$. It is easy to check that $(1+xt^{-1}+\frac{1}{2}x^2 t^{-2})\cdot e_0\in \mathcal{M}_{(1 0^{\ell-1})}$. Conversely, let $g(t)\cdot e_0=(1+xt^{-1}+yt^{-2})\cdot e_0\in\mathcal{M}_{(1 0^{\ell-1})}$. Let $\iota(g(t))=1+xt^{-1}+y't^{-2}$. Since $g(t)=g(t)^{-T}=(\iota(g(-t)))^T$, $y=(y')^T$. Then $y$ and $y'$ are adjoint each other under the symmetric form whose matrix is $I$. Note that $\rk y=\rk y'=1$. If $\im y\neq \im y'$, then $\rk x^2=\rk y+\rk y'=2$, a contradiction. Hence $\im y=\im y'$. By Lemma \ref{lemma:adjointmap}, $y'=y$ or $y'=-y$. By (\ref{equation:conditiony'}), $x^2=y+y'$ and hence $y'=y$. By (\ref{equation:conditionsD}), $x^T +x=0$ and $x^T x +y^T +y=0$. Then $y+y'=x^2=y+y^T$, so $y=y'=y^T$. Therefore,
	$g(t)=1+xt^{-1}+yt^{-2}=1+xt^{-1}+\frac{1}{2}x^2 t^{-2}.$
\end{proof}
\begin{prop}
\label{prop:fiber2}
	For $x\in\mathcal{N}_{\mathfrak{p}}$, write $x$ as in (\ref{equation:xy type D}). Then
	\begin{equation}\label{fiber2}
	\pi^{-1}(x)_{\rm red }\cong\{D\in \mathfrak{so}_{2\ell+1}\mid Du=0, D^2=0\}.\end{equation}
	In particular, $\pi^{-1}(0)_{\rm red}$ is isomorphic to the maximal order 2 nilpotent variety  in $\mathfrak{so}_{2\ell+1}$, and
	\[\dim\pi^{-1}(0)_{\rm red}= \begin{cases}
	\ell^2 & \quad \text{if } \ell \text{ is even;}\\
	\ell^2 -1 & \quad \text{if } \ell \text{ is odd.}\\
	\end{cases}\]
\end{prop}
\begin{proof}
	Under the realization $\mathfrak{so}_{2\ell+2}(I)$ and $\mathrm{SO}_{2\ell+2}(I)$, and the diagram automorphism $\sigma_0$, we have that $(1+xt^{-1}+yt^{-2})\cdot e_0 \in \mathcal{M}$ if and only if $w_0 y w_0=y$ and the conditions (\ref{equation:conditionsD}) hold. Set $z=y-\frac{1}{2}x^2$, these conditions are equivalent to
	\begin{equation}\label{equation:z D}
	z=\begin{pmatrix}
	0 & \\
	& D
	\end{pmatrix},\quad D\in \mathfrak{so}_{2\ell+1},\quad D^2=0,\quad Du=0,
	\end{equation}  
	where $u$ is given in (\ref{equation:xy type D}).
	Hence the isomorphism (\ref{fiber2}) holds.  In particular when $x=0$, 	$\pi^{-1}(0)_{\rm red}\cong \{D\in \mathfrak{so}_{2\ell+1}\mid D^2=0\}$ which is $\overline{\mathcal{O}}_{[2^k 1^{2\ell-2k+1}]}$ in $\mathfrak{so}_{2\ell+1}$ where $k$ is the maximal even integer. By the dimension formula, cf.\,\cite[Corollary 6.1.4]{CM},
	\[\dim\pi^{-1}(0)_{\rm red}= \begin{cases}
	\dim \mathcal{O}_{[2^\ell 1]}=\ell^2 & \quad \text{if } \ell \text{ is even;}\\
	\dim \mathcal{O}_{[2^{\ell-1} 1^3]}=\ell^2 -1 & \quad \text{if } \ell \text{ is odd}\\
	\end{cases}\]
	as desired.
\end{proof}

	Similar to the case $(A_{2\ell-1},2)$, it is not obvious to see if there are finitely many $K$-orbits in $\mathcal{M}_{(1^j 0^{\ell-j})}$. If $g(t)=1+xt^{-1}+(z+\frac{1}{2}x^2)t^{-2}$ such that $g(t)\cdot e_0\in\mathcal{M}_{(1^j 0^{\ell-j})}$, then $g(t)$ satisfies (\ref{equation:z D}).
	If the action of $K$ on the following anti-commuting nilpotent variety
	$$\{(x,z)\in \mathfrak{so}_{2\ell+2}(I)\times \mathfrak{so}_{2\ell+2}(I)\mid xz+zx=0, x, z \text{ nilpotent}\}$$
	by diagonal cojugation has finitely many orbits, then there are finitely many $K$-orbits in $\mathcal{M}_{(1^j 0^{\ell-j})}$.

\subsection{Theorem \ref{theorem:cellandorbit} for the field of positive characteristic}
In this subsection, we make a remark regarding Theorem  \ref{theorem:cellandorbit} when the filed $\mathbb{C}$ is replaced by an algebraically closed field $\mathrm{k}$ of positive characteristic $p$. 
\begin{theorem}
\label{main_thm_positve_char}
 Theorem \ref{theorem:cellandorbit} holds for the field $\mathrm{k}$ of characteristic $p$,  when 
 $$
 \begin{cases}
 p\geq 3 \quad  \text{ if } (X_N,r)=(A_{2\ell}, 2)\\
 p\geq 5 \quad \text{ if } (X_N,r)=(A_{2\ell-1}, 2)\\
 p\geq 3\quad  \text{ if }(X_N,r)=(D_{\ell+1}, 2) .
 \end{cases}
$$
\end{theorem}
\begin{proof}
Suppose that $p>2$. Given any element $L\in \mathcal{M}$, set $x=\pi(L)\in \mathfrak{p}$. 
When $(X_N,r)=(A_{2\ell}, 2)$, by the proof of Theorem \ref{theorem: isoA2l},   $x^2=0$. When $(X_N,r)=(A_{2\ell-1}, 2)$, by the proofs of Theorem \ref{theorem: isoA2l-1} and Theorem \ref{theorem:pi A2l-1}, $x^4=0$. Recall that $x\in \mathfrak{p}$ if and only if $x$ is self-adjoint with respect to a non-degenerate symmetric form (resp. symplectic form) when $\mathfrak{g}=A_{2\ell}$ (resp. $A_{2\ell-1}$).   
Under our assumption on the characteristic $p$, by the similar proof of \cite[Lemma 1.9]{Ja} for $x\in \mathfrak{p}$, we can find a nilpotent matrix $y$ in $\mathfrak{p}$ with the same order of $x$ and $h\in \mathfrak{k}$ such that $\{x,y,h\}$ is a $\mathfrak{sl}_2$-triple.  By \cite[Theorem 5.4.8]{Ca},  with the assumption on $p$, as a $\mathfrak{sl}_2$-representation, $V$ is completely reducible, i.e. we still have the decomposition $(8)$. Then  by the same argument as in Theorem \ref{theorem:classify}, all possible partitions of $x$ are exactly those that appear in Theorem \ref{theorem:cellandorbit}. 
When $(X_N,r)=(D_{\ell+1}, 2)$,  by the proof of Theorem \ref{theorem:piD}, $x^3=0$. When $p>2$, by \cite[Theorem 1.6]{Ja}, $x$ is either 0 or has partition $[31^{2\ell-1}]$, i.e. those that appear in  Theorem \ref{theorem:cellandorbit}.  Thus, all results in Section 4.1-4.3 remain true under our assumption on $p$. 
\end{proof}
We expect that Theorem \ref{theorem:cellandorbit} is true for any $p>2$.   Theorem \ref{theorem:cellandorbit}  relies on the classification theorem of nilpotent orbits in $\mathfrak{p}$. In fact, we expect Theorem  \ref{theorem:classify} and Theorem \ref{theorem:SO_{2n+1}} hold for any field $\mathrm{k}$ when the characteristic $p>2$. The reason is that the classification of nilpotent orbits in classical Lie algebra remains the same if $p>2$, see a proof in \cite[\S1.6-1.12]{Ja}.  A similar proof for the classification of nilpotent orbits in $\mathfrak{p}$ should also carry over when $p>2$. 

\section{Applications}
\label{sect:applications}

In this section, we describe some applications to the geometry of order 2 nilpotent varieties in the certain classical symmetric spaces. 

Let $\langle  , \rangle $ be a symmetric or symplectic non-degenerate bilinear form on a vector space $V$.  Recall that $\mathcal{A}$ is the space of all self-adjoint linear maps with respect to $\langle  , \rangle $.
Set $\mathcal{N}_{\mathcal{A},2 }$ denote the space of all nilpotent operators $x$ in $\mathcal{A}$ such that $x^2=0$.   If $\langle , \rangle$ is symmetric and $\dim V= 2n+1$, then ${\rm SO}_{2n+1}$-orbits in $\mathcal{N}_{\mathcal{A},2 }$ are classified by the partitions $[2^j 1^{2n+1-2j}]$ with $0\leq j\leq n$;  if $\langle , \rangle$ is symplectic and $\dim V= 2n$, then ${\rm Sp}_{2n}$-orbits in $\mathcal{N}_{\mathcal{A},2 }$ are classified by the partitions $[2^{2j} 1^{2n-2j}]$ with $0\leq j\leq \lfloor\frac{n}{2}\rfloor$.

\begin{theorem}
\label{thm_normal}
Assume that $\langle, \rangle $ is symplectic or symmetric and $\dim V$ is odd. Then 
any order 2 nilpotent variety  in $\mathcal{A}$ is normal.  
\end{theorem}
\begin{proof}
By Theorem \ref{theorem: isoA2l} and Theorem \ref{theorem: isoA2l-1},  for any order 2 nilpotent variety $\overline{\mathcal{O}}$ in $\mathcal{A}$, $\overline{\mathcal{O}}$ is isomorphic to $\overline{\mathcal{M}}_{\bar{\lambda}}:= \overline{\mathcal{G}r}_{\bar{\lambda}}\cap  \mathcal{G}r_0^-$ for a small dominant weight $\bar{\lambda}$ of $H$.  Note that $\overline{\mathcal{M}}_{\bar{\lambda}}$ is an open subset of the twisted Schubert variety $\overline{\mathcal{G}r }_{\bar{\lambda}}$ and $\overline{\mathcal{G}r }_{\bar{\lambda}}$ is a normal variety (cf.\,\cite[Theorem 0.3]{PR}). It follows that $\overline{\mathcal{O}}$  is also normal. 
\end{proof}
In fact, when $\langle , \rangle$ is symplectic,  any nilpotent variety in $\mathcal{A}$ is normal,  see \cite{Oh}.  In {\it loc.cit.}, Ohta also showed that not all nilpotent varieties are $\mathcal{N}_\mathfrak{p}$ is normal, when $\langle  ,\rangle$ is symmetric.  When $\langle  ,\rangle$ is symmetric and $\dim V$ is odd,  this theorem seems to be new.

\begin{remark}
Theorem \ref{thm_normal} is true for any field $k$ of characteristic $p>2$, as one can see that the classification theorem in Section \ref{sect_3} still holds for order 2 nilpotent orbits, and the arguments in  Theorem \ref{theorem: isoA2l}, Theorem \ref{theorem: isoA2l-1} applies as well. See the discussions in the proof of Theorem \ref{main_thm_positve_char}.  The same remark applies to the following Theorem \ref{thm_duality} and Theorem \ref{thm_locus}
\end{remark}

For any variety $X$, let ${\rm IC}_X$ denote the intersection cohomology sheaf on $X$. The perverse sheaf ${\rm IC}_X$ captures the singularity of the variety $X$.  
For any $x\in X$,  we denote by $\mathscr{H}^k_x( {\rm IC}_X )$ the $k$-th cohomology of the stalk of ${\rm IC}_X$ at $x$.

\begin{theorem}
\label{thm_duality}
\begin{enumerate}
\item When $\langle  ,\rangle$ is symmetric and $\dim V=2n+1$,   for any $0\leq j\leq n$, let $\mathcal{O}_j$ denote the nilpotent orbit in $\mathcal{A}$ associated to the partition $[2^j 1^{2n+1-2j}]$ and let 
 $\mathcal{O}'_{j}$ denote the nilpotent orbit in $\mathfrak{sp}_{2n }$ associated to the partition  $[2^{j} 1^{2n-2j}]$,  we have 
\begin{equation*}
\label{dimension_orbit1}
  \dim  \mathcal{O}_{j}= \dim  \mathcal{O}'_{j}  = j(2n+1-j) .\end{equation*}
Moreover,  for any $x\in \mathcal{O}_{[2^i 1^{2n+1-2i}]}$ and $x'\in \mathcal{O}'_{[2^i 1^{2n-2i}]} $, and for any $k\in \mathbb{Z}$, 
\[\dim  \mathscr{H}_x^k({\rm IC}_{ \overline{ \mathcal{O}}_{ j } } ) = \dim  \mathscr{H}_x^k({\rm IC}_{ \overline{ \mathcal{O}'}_{j} }  )   . \]

\item  When $\langle  ,\rangle$ is symplectic and $\dim V=2n$,  for any $0\leq j \leq \lfloor\frac{n}{2}\rfloor$, let $\mathcal{O}_{2j}$ denote the nilpotent orbit in $\mathcal{A}$ associated to the partition $[2^{2j} 1^{2n-4j}]$  and let $\mathcal{O}'_{2j}$ denote the nilpotent orbit in $\mathfrak{so}_{2n+1}$ associated to the partition  $[2^{2j} 1^{2n+1-4j}]$, we have 
\begin{equation*} 
\label{dimension_orbit2}
 \dim  \mathcal{O}_{2j}= \dim  \mathcal{O}'_{2j}  = 4j(n-j), \end{equation*}
Moreover,  for any integer $0\leq i\leq j$, $x\in  \mathcal{O}_{2i}$, $x'\in  \mathcal{O}'_{2i}$, and for any $k\in \mathbb{Z}$, we have 
\[ \dim  \mathscr{H}_x^k({\rm IC}_{ \overline{ \mathcal{O}}_{2j} }  ) = \dim  \mathscr{H}_x^k({\rm IC}_{ \overline{ \mathcal{O}'}_{2j} }  )   . \]
\end{enumerate}
\end{theorem}
\begin{proof}
We first prove part 1). 
By Theorem \ref{theorem:dimension} and \cite[Corollary 6.1.4]{CM},  it is easy to verify $ \dim  \mathcal{O}_{j}= \dim  \mathcal{O}'_{j}  = j(2n+1-j) $. 
  By Theorem \ref{theorem: isoA2l},   $ \overline{ \mathcal{O}}_{j }$ can be embedded into an open subset in the twisted affine Schubert variety $\overline{\mathcal{G}r}_{\omega_j }$ associated to $(\mathrm{SL}_{2n+1} , \sigma)$.  On the other hand, in view of \cite{AH},  $ \overline{ \mathcal{O}'}_{j}$ can be embedded into the untwisted affine Schubert variety $\overline{\Gr}^{\omega_j}_{\mathrm{Sp}_{2n}}$ in the affine Grassmannian $\overline{\Gr}_{\mathrm{Sp}_{2n}}$ of $\mathrm{Sp}_{2n}$.  Set 
  \[ \mathcal{F}=  {\rm IC}_{ \overline{ \mathcal{O}}_{ j } } [-\dim  \overline{ \mathcal{O}}_{ j }    ] , \quad  \text{ and }   \mathcal{F}'= {\rm IC}_{\overline{ \mathcal{O}'}_{ j }} [-\dim  \overline{ \mathcal{O}'}_{ j } ]  . \]
  By purity vanishing property of  intersection cohomology sheaf of Schubert varieties (cf.\,\cite{KL}),  $\mathscr{H}_x^k(\mathcal{F})= \mathscr{H}_{x'}^k(\mathcal{F}')=0$ when $k$ is odd. Equivalently, 
  \[ \mathscr{H}_x^k({\rm IC}_{ \overline{ \mathcal{O}}_{ j } } ) =   \mathscr{H}_x^k({\rm IC}_{ \overline{ \mathcal{O}'}_{j} }  )=0 \]
for any odd integer $k$, as $\dim \overline{ \mathcal{O}}_{ j } =\dim \overline{ \mathcal{O}'}_{ j } $ is even.
  
 Note that the affine Grassmannian $\Gr_{\mathrm{Sp}_{2n}}$ and the twisted affine Grassmannian $\mathcal{G}r_{{\rm SL}_{2n+1} }$ have the same underlying affine Weyl group.  Applying the results in \cite{KL}, the polynomials $\sum  \dim  \mathscr{H}_x^{2k}( \mathcal{F} ) q^k$ and $\sum   \mathscr{H}_x^{2k}( \mathcal{F}' )   q^k  $ are 
both equal to the same  Kazhdan-Lusztig polynomial $P_{\omega_i, \omega_j} (q)$ for the affine Weyl group of $\mathfrak{so}_{2n+1}$.  It follows that 
 \[ \dim \mathscr{H}_x^k({\rm IC}_{ \overline{ \mathcal{O}}_{ j } } ) = \dim   \mathscr{H}_x^k({\rm IC}_{ \overline{ \mathcal{O}'}_{j} }  ) \]
for all even integer $k$. Alternatively, one can see these two polynomials are equal, as they both coincide with the jump polynomial of the Brylinsky-Kostant filtration on the irreducible representation $V_{\omega_j}$ of $H$, see \cite{Bry,Zh}.

For the second part of the theorem, the proof is almost the same, except that by Theorem \ref{theorem: isoA2l-1}, $ \overline{ \mathcal{O}}_{2j}$ can be openly embedded into the twisted affine Schubert variety $\overline{\mathcal{G}r}_{\omega_{2j}} $ associated to $(\mathrm{SL}_{2n} ,  \sigma)$, and  $ \overline{ \mathcal{O}}'_{2j}$ can be openly embedded into  the affine Schubert variety $\overline{ \Gr}^{\omega_{2j}} _{\mathrm{Spin}_{2n+1} }$.  
\end{proof}
Part 1) of this theorem was due to Chen-Xue-Vilonen \cite{CVX} by different methods.  This theorem shows that there is a natural bijection between  order 2 nilpotent varieties in $\mathcal{A}$ and order 2 nilpotent varieties in its dual classical Lie algebras, such that they share similar geometry and singularities. 

We now describe another application. 

\begin{theorem}
\label{thm_locus}
If $\langle  , \rangle $ is symplectic, then the smooth locus of any order 2 nilpotent variety in $\mathcal{A}$ is the open nilpotent orbit.
\end{theorem}
\begin{proof}
Let $\overline{\mathcal{O}}$ be any order 2 nilpotent variety in $\mathcal{A}$. By Theorem \ref{theorem: isoA2l-1},  $\overline{\mathcal{O}}$ can be openly embedded into a twisted Schubert variety $\overline{\mathcal{G}r}_{\bar{\lambda}}$  with $\bar{\lambda}$ small, in the twisted affine Grassmannian $\mathcal{G}r_{\mathrm{SL}_{2n} }$.  Then this theorem follows from  \cite[Theorem 1.2]{BH}.
\end{proof}

\end{document}